\documentclass[12pt]{amsart}
\usepackage{graphicx}
\usepackage{subfigure}
\usepackage{latexsym}
\usepackage{amsmath}
\usepackage{amssymb}
\usepackage{mathrsfs}
\usepackage{amscd}
\usepackage{verbatim}

\pdfoutput=1

 \newtheorem{theorem}{Theorem}[section]
 \newtheorem{Def}[theorem]{Definition}
 \newtheorem{Prop}[theorem]{Proposition}
 \newtheorem{Lem}[theorem]{Lemma}
 \newtheorem{Cor}[theorem]{Corollary}

 \newtheorem{Exa}[theorem]{Example}

 \numberwithin{equation}{section}

\begin{document}

\title{Lipschitz equivalence  of self-similar sets and hyperbolic boundaries}

\author{Jun Jason Luo} \address{Department of Mathematics, The Chinese University of Hong Kong,
Hong Kong, and Department of Mathematics, Shantou University,  Shantou, Guangdong 515063, P.R. China} \email{ luojun2011@@yahoo.com.cn}

\author{Ka-Sing Lau} \address{Department of Mathematics, The Chinese
University of Hong Kong, Hong Kong} \email{ kslau@@math.cuhk.edu.hk}

\thanks{The research is supported by the HKRGC grant,  the Focus Investment Scheme of CUHK, and  the NNSF of China (no. 10871065)}


\begin{abstract}
 In \cite{Ka03} Kaimanovich introduced the concept of {\it   augmented tree}  on the symbolic space of a self-similar set. It is hyperbolic in the sense of Gromov, and it was shown in \cite {LaWa09} that  under the open set condition, a self-similar set can be identified with the hyperbolic boundary of the tree. In
the paper, we investigate in detail a class of  {\it simple}
augmented trees and the Lipschitz equivalence of such trees. The
main purpose is to use this to study the Lipschitz equivalence
problem of the totally disconnected  self-similar sets which has
been undergoing some extensive development recently.
\end{abstract}

\maketitle

\bigskip

\begin{section}{\bf Introduction}

Two compact metric spaces $(X,d_X)$ and $(Y,d_Y)$ are said to be
{\it Lipschitz equivalent}, and denote by $X\simeq Y$, if there is a
bi-Lipschitz map  $\sigma$ from $X$ onto $Y$, i.e., $\sigma$ is a
bijection and  there is a constant $C>0$ such that
$$
C^{-1}d_X(x,y)\leq d_Y(\sigma(x),\sigma(y))\leq Cd_X(x,y)\quad  \quad \forall  \ x,y\in X.
$$
It is easy to see that if $X \simeq Y$, then $\dim_HX=\dim_HY =s $, where
$\dim_H$ denotes the Hausdorff dimension. A more intensive study of this was due to Cooper and Pignartaro  \cite{CoPi88} in the late 80's, they showed that for certain Cantor sets $X, Y$ on ${\Bbb R}$, the Lipschitz equivalence implies that there exists a bi-Lipschitz map $\sigma$ and  a $\lambda >0$ such that ${\mathcal H}^s (\sigma (E)) = \lambda{\mathcal H}^s (E)$. In another consideration,  Falconer and Marsh
\cite{FaMa89} proved that for quasi-self-similar circles, they are Lipschitz equivalent if and only if they have the same Hausdorff
dimension.

The recent interest of the Lipschtiz equivalence is due to the path
breaking study of Rao, Ruan and Xi \cite{RaRuXi06} on a question of
David and Semmes \cite{DaSe97} on certain special self-similar set
on ${\Bbb R}$ to be dust-like (see Example \ref {example1}). They
observed a graph directed relationship in the underlying iterated function system (IFS), and made use of this to construct the needed
bi-Lipschitz map. There is a number of generalizations on ${\Bbb R}$
and ${\Bbb R}^d$ for the totally disconnected self-similar sets
with uniform contraction ratio or with logarithmically commensurable contraction ratios (\cite{DeHe10}, \cite{RuWaXi12},
\cite{XiRu07}, \cite{XiXi10}, \cite{XiXi11}). For certain Cantor-type sets on ${\Bbb R}^d$  with non-equal contraction ratios, Rao,
Ruan and Wang \cite{RaRuWa10} had an elegant algebraic criterion for
them to be Lipschitz equivalent,  which improved a  condition of
Falconer and Marsh in \cite{FaMa92}. Other considerations can be
found in \cite{LlMa10}, \cite{MaSa09}, \cite{Xi04}.

So far the investigation of the Lipschitz equivalence is very much restricted on a few  special self-similar sets. In this paper, we will provide a broader framework to study the  problem through the concept of {\it augmented (rooted) tree}.
 For an IFS $\{S_i\}_{i=1}^m$ of contractive similitudes on ${\Bbb R}^d$ (assume equal contraction ratio in the present situation) and the associated self-similar set $K$,  we use $X=\bigcup_{n=0}^{\infty}\Sigma^n$, $\Sigma = \{1,\dots , m\}$ to denote the symbolic space. Then $X$ has a natural graph structure, and we denote the edge set by ${\mathcal E}_v$ ($v$ for vertical). We define a horizontal edge for a pair $({\bf u}, {\bf v})$  in $X\times X$ if ${\bf u}, {\bf v} \in \Sigma^n$ and $S_{\bf u}(K) \cap S_{\bf v}(K) \not = \emptyset$, and denote this set of edges by ${\mathcal E}_h$. The {\it augmented tree} is defined as the graph $(X, {\mathcal E})$ where $ {\mathcal E} = {\mathcal E}_v \cup {\mathcal E}_h$.

  Such augmented tree was first introduced by Kaimanovich \cite{Ka03} on the Sierpinski gasket in order to incorporate the intersection of the cells to the  symbolic space, and was developed by Lau and Wang \cite{LaWa09} to general self-similar sets. It was proved that if an IFS satisfies the open set condition (OSC), then the augmented tree is {\it hyperbolic} in the sense of Gromov. There is a hyperbolic metric $\rho$ on $X$, which induces a {\it hyperbolic boundary}  $(\partial X, \rho)$. The hyperbolic boundary is shown to be homeomorphic to  $K$;  moreover under certain mild condition, the homeomorphism is  actually a H\"older equivalent map. This setup is used to study the random walks on $(X, {\mathcal E})$ and their Martin boundaries \cite{LaWa11}.

Based on this, our approach to the Lipschitz equivalence of the
self-similar sets is to lift the consideration to the augmented
trees $(X, {\mathcal E})$. We define a {\it horizontal connected
component} of  $X$ to be the maximal connected horizontal subgraph
$T$ in some level $\Sigma^n$. Let ${\mathscr C}$ be the set of all
horizontal connected components of $X$.   For $T\in {\mathscr C}$,
we use $T\Sigma$ to denote the set of offsprings of $T$, and
consider $T\cup T\Sigma$ as a subgraph in $X$.  We say that $T,
T'\in {\mathscr C}$ are equivalent if they are graph isomorphic. We
call $X$ {\it simple} if there are finitely many equivalence
classes. Under this condition,  we can define an incidence matrix
$$
A=[a_{ij}]
$$
for the equivalence classes as follows:  choose any component $T$
belonging to the class ${\mathscr T}_i$, and let
$Z_{i1},\dots,Z_{i\ell}$ be the connected components of the descendants
of $T$. The entry $a_{ij}$ denotes the number of $Z_{ik}$ that
belonging to the class ${\mathscr T}_j$.

It is shown that a simple augmented tree $X$ is always
hyperbolic, and the relationship of the hyperbolic boundary and the self-similar set is analogous to the case with OSC (Propositions \ref{th3.4}, \ref{th3.8}). Our basic theorem, to put it into a simple statement,  is (assuming the IFS has equal contraction ratio):

\begin{theorem}\label{th1}
Suppose the augmented tree $(X,{\mathcal E})$ is simple and the
corresponding incidence matrix  $A$ is primitive (i.e., $A^n>0$ for some $n>0$), then $\partial (X,{\mathcal E}) \simeq\partial (X,
{\mathcal E}_v)$.
\end{theorem}

 We call a self-similar set $K$ {\it dust-like}  if it satisfies $S_i(K) \cap S_j(K)  = \emptyset $ for $i \not = j$. By reducing the Lipschitz equivalence on the trees in Theorem \ref{th1} to the self-similar sets (Proposition \ref{th3.8}, Theorem \ref {th3.9}), we have

\begin{theorem} \label {th2} (i) If in addition to the condition in Theorem \ref{th1}, the IFS satisfies some mild condition (H) (see Section 2), then $K$ is Lipschitz equivalent to a dust-like self-similar set.

 \medskip
 (ii) If $K$ and $K'$ are as in (i) and the two IFS's  have the same number of similitudes and the same contraction ratio, then they are Lipschitz equivalent.
 \end{theorem}

The proof of Theorem \ref{th1} depends on constructing a {\it
near-isometry} between the augmented tree $(X, {\mathcal E})$ and
$(X, {\mathcal E}_v )$.  The crux of the construction is to use a
technique of {\it rearrangement} of edges (Section 4), which is
based on an idea of Deng and He in \cite{DeHe10}. Actually  we prove
a less restrictive form of Theorem \ref{th1} (Theorem \ref{th3.6})
in terms of {\it rearrangeable matrices}. Theorem \ref{th1} follows
from another theorem (Theorem \ref{th3.6'}) that the primitive property implies rearrangeability .

 We  will provide an easy way to check an augmented tree being simple (Lemma \ref {th5.0}), which is more efficient to
apply to various examples than the graph directed systems that were used in the previous studies
 (\cite{DeHe10}, \cite{RaRuXi06}, \cite{XiXi10}).  Theorem \ref{th2} essentially covers all the known cases so far,
 it also covers  some new classes of IFS's that have overlaps and rotations (see Section 5).   Moreover the theory can be extended
 from the self-similar IFS to the  self-affine IFS: it is easy to see that the notion of augmented tree and the
results on such tree remain unchanged. For self-affine sets on
${\Bbb R}^d$, we can still establish the Lipschitz equivalence, making use
of a device in \cite {HeLa08} which replaces the Euclidean distance
by an ultra-metric adapted to the underlying self-affine system (see
Theorem \ref{th3.12}).

The paper is organized as follows. In Section 2, we recall some
well-known results on hyperbolic graphs and set up the augmented
tree. In Section 3, we introduce the notion of {\it simple}
augmented tree,  and derive its basic properties.  Theorem \ref{th1}
is stated there, and Theorem \ref{th2} together with other consequences is
proved. The proof of Theorem \ref{th1} and the involved concept of
rearrangement are given in Section 4.  In Section 5, we provide
several new examples to illustrate our results; some concluding remarks
and open questions are given in Section 6.
\end{section}

\bigskip

\begin{section}{\bf  Preliminaries}

Let $X$ be a countably infinite set, we say that $X$ is a {\it
graph} if it is associated with a symmetric subset  ${\mathcal E}$
of $(X \times X) \setminus \{(x,x) : \ x \in X\}$;  we call  $x \in
X$ a {\it vertex}, $(x, y) \in {\mathcal{E}}$ an {\it edge}, which
is more conveniently denoted by $x \sim y$ (intuitively, $x, y$ are
neighborhoods to each other).  By a {\it path} in $X$ from $x$ to
$y$, we mean a finite sequence $x = x_0, x_1, \dots , x_n=y$ such
that $x_i \sim x_{i+1}, i =0, \dots, n-1$. We always assume that the
graph $X$ is connected, i.e., there is a path joining any two
vertices $x, y \in X$. We call $X$ a {\it tree} if the path between
any two points is unique. We equip a graph $X$ with an
integer-valued metric $d(x,y)$, which is the minimum among the
lengths of the paths from $x$ to $y$; the corresponding geodesic
path is denoted by $\pi(x,y)$ and its length by $|\pi(x,y)|
(=d(x,y))$. Let $ o \in X$ be a fixed point in $X$ and call it the
{\it root} of the graph. We use $|x|$ to denote $d(o,x),$ and say
that $x$ belongs to the $n$-th level of the graph if $d(o,x)=n$.

 The notion of hyperbolic graph  was introduced by Gromov (\cite{Gr87}, \cite{Wo00}).  First we define the Gromov product
 of any two points $x,y
\in X$ by
$$
|x \wedge y| =\frac{1}{2}(|x|+|y|-d(x,y)).
$$
We say that $X$ is {\it hyperbolic} (with respect to $o$) if there
is $\delta >0$ such that
$$
 |x \wedge y| \geq \min\{|x \wedge z|,  |z\wedge y|\}-\delta \qquad \forall \ \  x,y,z \in X.
$$
Note that this is equivalent to a more geometric characterization:
there exists a $\delta'$ such that for any three points in $X$, the
geodesic triangle is $\delta'$-{\it thin}: any point on one side of
the triangle has distance less than $\delta'$ to some point on one of the
other two sides.

For a fixed $a>0$ with $a'=\exp(\delta a)-1<\sqrt{2}-1$, we define
an {\it ultra-metric} ${\rho_a}(\cdot,\cdot)$ on $X$ by
\begin{equation} \label{eq2.0}
{\rho_a}(x,y)=\left\{
\begin{array}{ll}
\exp(-a |x \wedge y| ) & \quad \textrm{if \ $x\ne y$},\\
0  & \quad \textrm{otherwise}.
\end{array} \right.
\end{equation}
Then
$$
{\rho_a}(x,y)\leq
(1+a')\max\{{\rho_a}(x,z),{\rho_a}(z,y)\} \qquad \forall \
x,y,z\in X,
$$
which is equivalent to  the path metric
$$
{\theta_a}(x,y)=\inf \Big\{ \sum_{i=1}^n{\rho_a}(x_{i-1},x_i): n
\geq 1,\ x = x_0,x_1,\dots,x_n=y, \ x_i\in X \Big\}.
$$
Since $\theta_a$ and $\rho_a$ determine the same topology as long as
$a'<\sqrt{2}-1$, we will use $\rho_a$ to replace $\theta_a$  for
simplicity.

\begin{Def}
The hyperbolic boundary of $X$ is defined as $\partial X=
\hat{X}\setminus X$ where $\hat{X}$ is the completion of $X$ under
$\rho_a$.
\end{Def}

The metric $\rho_a$ can be extended onto $\partial X$, and under which
$\partial X$ is a compact set.  It is often useful to identify $\xi
\in \partial X$ with a {\it geodesic ray} in $X$ that converge to
$\xi$, i.e., an infinite path $\pi[x_1,x_2,\dots]$ such that
$x_i\sim x_{i+1}$ and any finite segment of the path is a geodesic.
It is known that two geodesic rays $\pi[x_1,x_2,\dots],
\pi[y_1,y_2,\dots]$ represent the same $\xi \in \partial X$ if and
only if $ |x_n\wedge y_n|\rightarrow\infty$ as $n\rightarrow\infty.$

Our interest is on the following tree structure introduced by
Kaimanovich which is used to study the self-similar sets (\cite
{Ka03}, {\cite {LaWa09}}).  For a tree $X$ with a root $o$, we  use
${\mathcal E}_v$ to denote the set of edges ($v$ for vertical). We
introduce  additional edges on  each level $\{x: \ d(o, x) =n\}, \ n
\in {\Bbb N}$ as follows. Let $x^{-k}$ denote the $k$-th ancestor of
$x$, the unique point in the $(n-k)$-th level that is joined by a
unique path.

\begin{Def} \label{de2.2}
Let $X$ be a tree with a root $o$. Let \  ${\mathcal E}_h \subset (X
\times X) \setminus \{(x,x): \ x\in X\}$ such that it is symmetric
and satisfies:
$$
(x,y)\in {\mathcal E}_h \quad \Rightarrow \quad |x|=|y| \ \ \text{and}\ \
\text{either}~ x^{-1}=y^{-1} ~\text{or}~ (x^{-1},y^{-1})\in
{\mathcal E}_h.
$$
We call elements in  ${\mathcal E}_h$  horizontal edges,  and
for ${\mathcal E}={\mathcal E}_v\cup{\mathcal E}_h$, $(X,{\mathcal
E})$ is called an {augmented tree}.
\end{Def}

Following  {\cite {LaWa09}},  we say that a  path $\pi(x,y)$ is
a {\it horizontal geodesic} if it is a geodesic and it consists of
horizontal edges only. It is called a {\it canonical geodesic}  if
there exist $u,v\in \pi(x,y)$ such that:

(i)\  $\pi(x,y)=\pi(x,u)\cup\pi(u,v)\cup\pi(v,y)$ with $\pi(u,v)$ a
horizontal path and $\pi(x,u),\\ \pi(v,y)$ vertical paths;

(ii)\ for any geodesic path $\pi^{\prime}(x,y)$,
$\text{dist}(o,\pi(x,y))\leq \text{dist}(o,\pi^{\prime}(x,y))$.

\noindent Note that condition (ii) is to require the horizontal part
of the canonical geodesic to be on the highest level.  The following
basic theorem was proved in \cite{LaWa09}:

\begin{theorem} {\label{th2.3}} Let $X$ be an augmented tree. Then

\medskip

 { (i)}  Let $\pi(x,y)$ be a canonical geodesic, then $
|x\wedge y| = l-h/2$, where $l$ and $h$ are the level
and the length of the horizontal part of the geodesic.

\medskip

{ (ii)}  $X$ is hyperbolic if and only if there exists a constant $k>0$ such that any  horizontal part of a  geodesic is bounded by $k$.
\end{theorem}

 The premier application of the augmented trees is to  use their hyperbolic boundaries to study  the self-similar sets. Throughout the paper, we assume a self-similar set $K$ is generated by an iterated function system (IFS) $\{S_i\}_{i=1}^m$ on ${\Bbb R}^d$ where
\begin{equation} \label{eq2.2}
S_i(x) = r R_i x+d_i, \qquad i=1, \dots, m
\end{equation}
with $0<r<1$, $R_i$ is an orthogonal matrix, and $d_i \in {\Bbb
R}^d$. It is well-known that $K$ satisfies $ K=\bigcup_{i=1}^m
S_i(K). $ Let $\Sigma = \{1, \dots, m\}$ and let
$X=\bigcup_{n=0}^{\infty} \Sigma^n$ be the symbolic space
representing the IFS (by convention, $\Sigma^0 = \emptyset$, and we
still denote it by $o$ ). For ${\bf u} = i_1\cdots i_n$, we use
$S_{\bf u}$ to denote the composition $S_{i_1} \circ \cdots \circ
S_{i_n}$.

Let ${\mathcal E}_v$ be the set of vertical edges corresponding to the
nature tree structure on $X$ with $o$ as a root.  In \cite{LaWa09},
a set of horizontal edges ${\mathcal E}_h$ is defined  as
\begin{equation*}
{\mathcal E}_h =\left\{({\mathbf u},{\mathbf v}): |{\mathbf u}| =
|{\mathbf v}|,  {\mathbf u} \ne {\mathbf v}\ \text{and} \
K_{\mathbf u}\cap K_{\mathbf v} \ne \emptyset \right\},
\end{equation*}
where $K_{\mathbf u} = S_{\mathbf u}(K)$. Let ${\mathcal
E}={\mathcal E}_v \cup {\mathcal E}_h$, then $(X,{\mathcal E})$ is
an augmented tree induced by the self-similar set.

If the IFS is strongly separated (i.e., $S_i(K)\cap
S_j(K) = \emptyset$ for $i \ne j$),  then  $K$ is called {\it dust-like}.  It is clear that in this case, ${\mathcal E}_h = \emptyset$,  and $\rho_a$
coincides with the natural metric on the symbolic space:
$$
\varrho (x, y) = \exp (-a \max\{k:  x_i = y_i,\ i \leq
k\}).
$$
In \cite {LaWa09}, it was proved that under the {\it open set
condition }(OSC), Theorem \ref {th2.3}(ii) implies that the above
augmented tree is hyperbolic, and the nature map $\Phi: \partial X
\to K$ is a homeomorphism. Moreover if in addition, the IFS
satisfies

\medskip

\noindent  \emph{condition (H): \ \  there exists a constant
$c>0$ such that for any integer $n\geq 1$ and words ${\mathbf
u,v}\in \Sigma^n$, }
\begin{equation} \label{eq2.4}
K_{\mathbf u}\cap K_{\mathbf v} = \emptyset \quad \Rightarrow
\quad \hbox {dist} (K_{\mathbf u},K_{\mathbf v})\geq cr^n.
\end{equation}
Then for $\alpha = -\log r/a$, $\Phi$ satisfies the following
H\"older equivalent property:
\begin{equation*}
C^{-1}|\Phi(\xi) - \Phi(\eta)|  \leq  \rho_a (\xi, \eta)^\alpha
\leq C|\Phi(\xi) - \Phi(\eta)| \quad \forall \ \xi, \eta \in
\partial X.
\end{equation*}

Condition (H) is satisfied by the standard self-similar sets, for example, the generating IFS has the OSC and all the parameters of the similitudes are
integers. However there are also examples that condition (H) is not
satisfied (see \cite{LaWa09} for an example such that the
similitudes involve irrational translations).

From Definition \ref{de2.2}, we see that the choice of the
horizontal edges for the augmented tree can be quite flexible, for
example we can use $K_{\mathbf u}\cap K_{\mathbf v}$ to have
positive dimension to define ${\mathcal E}_h$. In this paper, we
will use another  setting by replacing $K$ with a bounded closed
invariant set $J$ (i.e., $S_i(J) \subset J$ for each $i$), namely
\begin{equation} \label {eq2.5}
{\mathcal E}_h =\left\{({\mathbf u},{\mathbf v}): |{\mathbf u}| =
|{\mathbf v}|,  {\mathbf u} \ne {\mathbf v}\ \text{and} \
J_{\mathbf u}\cap J_{\mathbf v} \ne \emptyset \right\}.
\end{equation}
We can take $J = K$ as before or in many situations, take $J= \overline U$ for the $U$ in the OSC  (see the examples in Section 5). The above statements on
the hyperbolicity of the augmented tree and the homeomorphism of the
boundary still valid by adopting the same proof; for the H\"older equivalence, we use the following modification of condition (H) for $J$, which will be used again in proving Proposition \ref{th3.8}.

\begin{Lem} \label{th2.4} Suppose the IFS in (\ref{eq2.2}) satisfies condition (H), then for any bounded closed invariant set $J$, there exists $c'>0$ and $k\geq 0$ such that for any $n \geq 0$ and ${\bf u}, {\bf v} \in \Sigma^n$,
\begin{equation*}
J_{\bf u} \cap J_{\bf v} = \emptyset \ \ \Rightarrow \ \ \hbox
{dist} (J_{\bf ui}, J_{\bf vj}) \geq c'r^n \qquad \forall \ {\bf
i},\  {\bf j} \in \Sigma^k.
\end{equation*}
\end{Lem}

\begin{proof}  Let $c$ be the constant in the definition of condition (H). For the bounded closed invariant set $J$, we have $K \subseteq J$ and   the Hausdorff distance $ d_H(K_{\bf i}, J_{\bf i}) \leq c_1r^k$ for all ${\bf i}\in \Sigma^k$.  In particular we choose $k$ such that $c_1r^k < c/3$.

Now if ${\bf u}, {\bf v} \in \Sigma^n, \ J_{\bf u} \cap J_{\bf v} = \emptyset$, then $K_{\bf u} \cap K_{\bf v} = \emptyset$, it follows from condition (H) that $\hbox { dist} (K_{\bf u}, K_{\bf v}) \geq cr^n$ for ${\bf u}, {\bf v} \in \Sigma^n$. Applying this and the above to $n+k$, we have
$$
\begin{aligned}
\hbox { dist} (J_{\bf ui}, J_{\bf vj}) & \geq  \hbox { dist} (K_{\bf ui}, K_{\bf vj}) -  d_H(K_{\bf ui}, J_{\bf ui})- d_H(K_{\bf vj}, J_{\bf vj})\\
&  \geq
c r^n -(2c/3)\ r^{n}\geq (c/3) r^n  \qquad \forall \ {\bf i},\  {\bf j} \in \Sigma^k.
\end{aligned}
$$
The lemma follows by taking $c' = c/3$.
\end{proof}

We remark that the augmented tree $(X, \mathcal E)$ depends on the
choice of the bounded invariant set $J$.  But under the OSC, the
hyperbolic boundary is the same as they can be identified with the
underlying self-similar set.

We conclude this section with the following simple relationship of
the totally disconnected self-similar set and the structure of the
augmented tree. The more explicit study of their Lipschitz equivalence
will be carried out in detail in the rest of the paper.  By a {\it
horizontal connected component}  of an augmented tree, we mean a
maximal connected horizontal subgraph on some level $\Sigma^n$ of $X$.

\begin{Prop} \label{additional prop}
Suppose the cardinality of any horizontal component in the augmented
tree  induced by the IFS in (\ref{eq2.2}) is uniformly bounded, then
the associated self-similar set $K$ is totally disconnected.

The converse is also true if the about IFS is on ${\Bbb R}^1$ and satisfies the OSC.
\end{Prop}

\begin{proof}
Suppose $K$ is not totally disconnected, then there is a connected
component $C\subset K$ contains more that one point. Note that for
any $n>0$, $K = \bigcup_{{\bf i}\in\Sigma^n} K_{\bf i}$. Let
$K_{{\bf i}_1} \cap C \not = \emptyset$. If $C\setminus K_{{\bf
i}_1} \not = \emptyset$,  then it is a relatively open set in $C$,
and as $C$ is connected,
$$
\partial_C (C \setminus K_{{\bf i}_1}) \cap   \partial_C (K_{{\bf i}_1}\cap C) \not = \emptyset.
$$
($\partial_C (E)$ means the relative boundary  of $E$ in $C$). Let $x$
be in the intersection,  there exists ${\bf i}_2 \in \Sigma^n$ such
that $ x \in K_{{\bf i}_1} \cap  K_{{\bf i}_2}$ and $K_{{\bf
i}_2}\cap (C \setminus K_{{\bf i}_1}) \not = \emptyset$.

Inductively, if \  $ \bigcup _{j=1}^kK_{{\bf i}_j}$ does not cover
$C$,  then we can repeat the same procedure to find ${\bf i}_{k+1}
\in \Sigma^n$ such that
$$
K_{{\bf i}_{k+1}} \cap  ( {\bigcup}_{j=1}^k K_{{\bf i}_j}) \not =
\emptyset \quad \hbox {and} \quad K_{{\bf i}_{k+1}}\cap (C \setminus
{\bigcup}_{j=1}^kK_{{\bf i}_j}) \not = \emptyset.
$$
Since $K = \bigcup_{{\bf i}\in\Sigma^n} K_{\bf i}$, this process
must end at some step, say $\ell$, and in this case $C \subset
{\bigcup}_{j=1}^\ell K_{{\bf i}_j}$. Since the diameter $|K_{{\bf
i}_j}| = r^n |K| \to 0$ as $n \to \infty$, $\ell$ must tend to
infinity, which contradicts the uniform boundedness of the horizontal connected components $\Sigma^n$.

For the converse,  we assume that the IFS is defined on ${\Bbb R}$.  Note that if $K \subset {\Bbb R}$ is totally
disconnected, then $\dim_H K=s<1$ \cite{Sc94}. Let  $\mu$  denote
the restriction of the $s$-Hausdorff measure on $K$. Without loss of
generality, we assume
 $\mu(K)=1$. Then it is well-known that for any point $x\in K$
and any $0< t< |K|$ (where $|K|$ denotes the diameter of $K$),
\begin{equation*}
C_1<\frac{\mu(B(x,t))}{t^s}<C_2,
\end{equation*}
where $C_1,C_2$ are constants independent of $x$ and $t$.

Suppose  ${\mathbf i}_1,{\mathbf i}_2,\dots,{\mathbf i}_k$ is a
finite sequence of distinct words in $\Sigma^n$ and is in a
horizontal connected component (we take $J= K$ for convenience),
i.e., $K_{{\mathbf i}_j}\cap K_{{\mathbf i}_{j+1}}\ne \emptyset$ for
$1\leq j\leq k-1$. Let $G$ be the smallest interval containing
$K_{{\mathbf i}_1},\dots,K_{{\mathbf i}_k}$. Let $x\in G \cap K$,
and let  $t= |G|$. Then the Hausdorff measure $\mu$ implies
$$
 \mu(B(x,t))\geq  kr^{ns}.
$$
This together with $ t= |G|\leq
kr^n |K|$ implies \begin{equation*}
\frac{k^{1-s}}{|K|^s} = \frac{kr^{ns}}{(kr^n
|K|)^s}\leq \frac{\mu(B(x,t))}{t^s}  \leq C_2.
\end{equation*}
Hence $k$ is uniformly bounded.
\end{proof}
\end{section}

\bigskip

\begin{section}{\bf  Lipschitz equivalence and the main theorems}

In this rest of the paper, unless otherwise stated, we will assume
the augmented tree $(X, {\mathcal E})$  is
associated with the symbolic space of the IFS $\{S_i\}_{i=1}^m$ in (\ref{eq2.2}),  and  ${\mathcal E} = {\mathcal E}_h \cup{\mathcal E}_v$ where ${\mathcal E}_h$  is defined by a fixed bounded closed invariant set $J$ as in (\ref {eq2.5}). We introduce a class of mappings between two hyperbolic graphs which plays a key role in the Lipschitz equivalence.

\begin{Def}
Let $X$ and $Y$ be two hyperbolic graphs and let $\sigma: X
\rightarrow Y$  be a bijective map. We say that $\sigma$ is a {
near-isometry} if there exists $c>0$ such that
$$
 \big ||\pi(\sigma(x),\sigma(y))|-|\pi(x,y)|\big| \leq c \qquad \forall  \ x, y \in X.
$$
\end{Def}

\noindent {\bf Remark}.   By checking $\pi (o, x)$, it is easy to show
that the above definition  implies  $\big | |\sigma (x)|
- |x| \big | \leq c + 3|\sigma (o)|+k$ where $k$ is the bound of the
horizontal geodesic in  Theorem \ref{th2.3}(ii).
Hence the above definition is equivalent to
$$
 \big||\sigma(x)| -|x|\big| <c, \quad \big ||\pi(\sigma(x),\sigma(y))|-|\pi(x,y)|\big| \leq c \qquad \forall  \ x, y \in X.
$$
(with different constant $c$)

\begin{Prop} \label{th3.3} Let $X$, $Y$ be two hyperbolic augmented trees that are equipped with the
hyperbolic metrics with the same parameter  $a$ (as in (\ref{eq2.0})). Suppose there exists a near-isometry  $\sigma : X \to Y$, then \ $\partial X \simeq \partial Y$.
\end{Prop}

\begin{proof} \ With the notation as in Theorem \ref{th2.3}(i), it follows  that for $x\ne y\in X$,
$$
|\pi(x,y)| =|x|+|y|-2l+h,  \quad |\pi(\sigma(x),\sigma(y))|=|\sigma(x)|+|\sigma(y)|-2l^{\prime}+h^{\prime}.
$$
  From the definition of $\sigma$ (and the remark), we have
$$
\big ||\sigma(x)| -|x|\big |, \ \big||\sigma(y)|-|y|\big |\leq c,\
|l-l^{\prime}|\leq 3c/2+k/2, \quad \hbox {and}\quad
|h-h^{\prime}|\leq k
$$
for some $k>0$  (where $k$ is the hyperbolic constant as in Theorem
\ref{th2.3}(ii)). By Theorem \ref {th2.3}(i),
$$
|x\wedge y| = l-h/2 \quad \hbox {and} \quad |\sigma(x)\wedge
\sigma(y)| =l^{\prime}-h^{\prime}/2.
$$
It follows that
$$
\big || \sigma(x)\wedge \sigma(y)|  - |x\wedge y|\big |\  = \
|l^{\prime}-h^{\prime}/2-l+h/2| \ \leq \ 3c/2+k \ : = k' .
$$
Let $\lambda= e^{ak^{\prime}},$ together with the definition of the
ultra-metric $\rho_a(x,y)= \exp(-a |x\wedge y|)$ in (\ref{eq2.0}),
we conclude that
$$
{\lambda}^{-1}{\rho_a}(x,y)\leq {\rho_a}(\sigma(x),\sigma(y))\leq
\lambda {\rho_a}(x,y) \qquad   \forall \ x, y \in X.
$$
By extending the metrics to the boundaries $\partial X,\  \partial
Y$, the above implies  $\sigma$ is a  bi-Lipschitz map from
$\partial X$ onto $\partial Y$.
\end{proof}

Let $\mathscr{C}$ be the set of all horizontal connected components
of $X$.   For $T\in \mathscr {C}$, we let $T\Sigma = \{{\bf u}i:
{\bf u} \in T, i \in \Sigma\}$ be the set of offsprings of $T$.
Note that if two distinct components $T, T' \in {\mathscr {C}}$ lie
in the same level, then $T\Sigma$ is not connected to $T'\Sigma$,
equivalently,
\begin{equation} \label{eq3.0}
\left({\bigcup}_{{\bf i}\in T\Sigma}J_{\bf i}\right)\cap
\left({\bigcup}_{{\bf j}\in T'\Sigma}J_{\bf j}\right)  = \emptyset.
\end{equation}
This follows easily from $S_{{\bf u}i}(J) \cap S_{{\bf v}j}
(J)\subset S_{\bf u} (J) \cap S_{\bf v} (J)= \emptyset$ for all
${\bf u} \in T,\  {\bf v} \in T',  \ i,j \in \Sigma $.

By regarding  $T\cup T\Sigma$ as a
subgraph in $X$.  We say that $T, T'\in \mathscr {C}$ are
{\it equivalent}, denoted by $T\sim T'$, if there exists a  graph
isomorphism
$$
g:\ T\cup T\Sigma \ \to \ T'\cup
T'\Sigma,
$$
 that is,  $g$ is a bijection such that $g$ and $g^{-1}$
preserve the vertical and horizontal edges.  It is easy to check that
$\sim$ is indeed an equivalence relation.  We use $[T]$ to denote
the equivalence class and call it a {\it connected class} determined
by $T$. Obviously, $\{o\}$ is the connected class determined by the
root $o$.

\begin{Def}
An augmented tree $X$ is called  { simple} (with respect to the
defining bounded closed invariant set $J$) if there are finitely
many connected classes, i.e., $\mathscr{C}/{\sim}$ is finite.
\end{Def}

\begin{Prop} \label{th3.4}
A simple augmented tree is always hyperbolic.
\end{Prop}

\begin{proof} Note that for each geodesic $\pi(x, y)$ in $X$, the horizontal part must be contained in a horizontal component of
the augmented tree. Since there are finitely many connected classes
$[T]$, and each $T$ contains finitely many vertices, it follows that
the horizontal part of  $\pi(x, y)$ is uniformly bounded, and hence hyperbolic by Theorem
\ref{th2.3}(ii).
\end{proof}

In the following we show that the hyperbolic boundary of a simple augmented tree is H\"older equivalent to the self-similar set, which is similar to the case in \cite{LaWa09} for the OSC.

\begin{Prop}\label{th3.8}
Let $\{S_i\}_{i=1}^m$ be an IFS satisfies condition (H) in
(\ref{eq2.4}), and assume that the corresponding augmented tree
$(X,{\mathcal E})$ is simple.  Then there exists a bijection
$\Phi:\partial X\to K$ satisfying the H\"older equivalent property:
\begin{align}\label{eq3.2}
C^{-1}|\Phi(\xi)-\Phi(\eta)|\leq \rho_a(\xi,\eta)^{\alpha}\leq
C|\Phi(\xi)-\Phi(\eta)|,
\end{align}
where $\alpha=-\log r/a$ and $C>0$ is a constant.
\end{Prop}

\begin{proof} The proof is essentially the same as in \cite {LaWa09}
by replacing $K$ with the invariant set $J$ in ${\mathcal E}_h$. We
sketch the main idea of proof here. For any geodesic ray $\xi=
\pi[{\mathbf u}_1,{\mathbf u}_2,\ldots]$,
 we define
$$
\Phi(\xi)=\lim_{n\rightarrow \infty} S_{{\mathbf u}_n}(x_0)
$$
for some $x_0\in J$. Then the mapping is well-defined  and is a
bijection (see Lemma 4.1, Theorem 4.3 in \cite {LaWa09}).

To show that $\Phi$ satisfies (\ref{eq3.2}),  let $\xi= \pi[{\mathbf
u}_0,{\mathbf u}_1,{\mathbf u}_2,\ldots], \ \eta= \pi[{\mathbf
v}_0,{\mathbf v}_1,{\mathbf v}_2,\ldots]$ be any two non-equivalent
geodesic rays in $X$. Then there is a canonical bilateral geodesic $\gamma$ joining $\xi$ and $\eta$:
$$
\gamma=\pi[\dots,{\mathbf u}_{n+1},{\mathbf u}_n,{\mathbf t}_1,\dots,{\mathbf t}_{\ell},{\mathbf v}_n,{\mathbf v}_{n+1},\dots]
$$
with ${\mathbf u}_n,{\mathbf t}_1,\dots,{\mathbf t}_{\ell},{\mathbf
v}_n\in \Sigma^n$. It follows that
$$
|S_{{\mathbf
u}_n}(x_0)-S_{{\mathbf v}_n}(x_0)|\leq (\ell+2) r^n|J|.
$$
Since $X$
is simple, $\ell$ is uniformly bounded (by Proposition \ref {th3.4}). Note that $\Phi(\xi)\in
J_{{\mathbf u}_k}$ and $\Phi(\eta)\in J_{{\mathbf v}_k}$ for all
$k\geq 0$, hence
$$
|\Phi(\xi)-S_{{\mathbf u}_n}(x_0)|,\quad |\Phi(\eta)-S_{{\mathbf v}_n}(x_0)|
\leq r^n|J|.
$$
Therefore
$$
\begin{aligned}
 & \ |\Phi(\xi)-\Phi(\eta)|\\
 \leq& \
|\Phi(\xi)-S_{{\mathbf u}_n}(x_0)|+|S_{{\mathbf
u}_n}(x_0)-S_{{\mathbf v}_n}(x_0)|+|\Phi(\eta)-S_{{\mathbf
v}_n}(x_0)|\\
 \leq & \  C_1 r^n.
\end{aligned}
$$
Since $\gamma$  is a bilateral canonical
geodesic, we have $|\xi\wedge\eta|=n-(\ell+1)/2$ and $\ell$ is
uniformly bounded. By using $\rho_a(\xi,\eta)=\exp(-a
|\xi\wedge\eta|)$, we see that
$$
|\Phi(\xi)-\Phi(\eta)|\leq C\rho_a(\xi,\eta)^{\alpha}.
$$

On the other hand,  assume that $\xi \ne \eta$. Since $\gamma$ is a
geodesic, it follows that $({\mathbf u}_{n+1},{\mathbf
v}_{n+1})\notin {\mathcal {E}}_h$, and hence $J_{{\mathbf
u}_{n+1}}\cap J_{{\mathbf v}_{n+1}}=\emptyset$. By Lemma \ref {th2.4}, there is $k$ (independent of $n$) such that
\begin{equation*}
J_{\bf u} \cap J_{\bf v} = \emptyset \ \ \Rightarrow \ \ \hbox
{dist} (J_{\bf ui}, J_{\bf vj}) \geq c'r^n \qquad \forall \ {\bf
i},\  {\bf j} \in \Sigma^k.
\end{equation*}
Referring to  $\gamma = \pi[\dots,{\mathbf u}_{n+1},{\mathbf
u}_n,{\mathbf t}_1,\dots,{\mathbf t}_{\ell},{\mathbf v}_n,{\mathbf
v}_{n+1},\dots]$, we have $\Phi(\xi) \in J_{{\bf u}_{n+k+1}}, \
\Phi(\xi) \in J_{{\bf v}_{n+k+1}}$.  It follows that
$$
|\Phi(\xi)-\Phi(\eta)|\geq \text{dist}(J_{{\mathbf u}_{n+k+1}},
J_{{\mathbf v}_{n+k+1}})\geq c' r^{n},
$$
and $|\Phi(\xi)-\Phi(\eta)|\geq c''\rho_a(\xi,\eta)^{\alpha}$
follows by the definition of  $\rho_a$ as the above.
\end{proof}

For a simple augmented tree $X$, we label the connected classes
as $\{{\mathscr{T}}_1,\dots,{\mathscr{T}}_r\}$ and introduce an
$r\times r$ incidence matrix
\begin {equation} \label{eq3.1}
A=[a_{ij}]_{r\times r}
\end{equation}
for the connected classes. The entries $a_{ij}$ are defined as
follows. For any $1\leq i\leq r$, take a horizontal connected
component $T$ in $X$ such that $[T]={\mathscr T}_i$. Let
$Z_{i1},\dots,Z_{i\ell}$ be the horizontal connected components
consisting of offsprings generated by $T$, i.e., $T\Sigma=
\bigcup_{k=1}^{\ell}Z_{ik}$, and define
$$
a_{ij}=\#\{k:1\leq k\leq \ell, [Z_{ik}]={\mathscr T}_j\}.
$$
Observe that $a_{ij}$ is independent of the choice of the components
in the equivalence classes. It is clear that for $T, T' \in
{\mathscr C}$,  $[T]=[T']$ implies $\#T=\#T'$. But the converse is
not true. As a direct consequence of the definition, we have

\begin{Prop}\label{th3.5}
Let ${\mathbf b}=[b_1,\dots,b_r]^t$
where $b_i=\#T$ where $[T]={\mathscr T}_i$, then $A{\mathbf
b}=m{\mathbf b}$.
\end{Prop}

The following theorem is for Lipschitz equivalence on the hyperbolic
boundaries,  it is the crucial step to establish the equivalence for
the self-similar sets.

\begin{theorem}\label{th3.6}
Suppose the augmented tree $(X,{\mathcal E})$ is
simple, and suppose the corresponding incidence
matrix $A$  is $(m, {\bf b})$-rearrangeable (where the $m$ and $\bf b$ are defined as in Proposition \ref {th3.5}).  Then there is a near-isometry between $ (X,{\mathcal E})$ and  $(X, {\mathcal E}_v)$, so that $\partial (X,{\mathcal E}) \simeq\partial (X, {\mathcal E}_v)$.
\end{theorem}

The notion of {\it rearrangeable matrix} is an important tool to
construct the near-isometry. Since the concept is a little complicated, we will introduce this in more detail, and prove
Theorem \ref{th3.6} together with the following theorem in the next section.

\begin {theorem} \label {th3.6'} If the incidence matrix $A$ is primitive, then  $A^k$ is $(m^k, {\bf b})$-rearrangeable for some $k>0$.  Consequently \ $\partial (X,{\mathcal E}) \simeq\partial (X, {\mathcal E}_v)$.
\end {theorem}

As a direct consequence, we have

\begin{Cor}\label {th3.7} Under the assumption on Theorem \ref {th3.6} (or Theorem \ref{th3.6'}),  then  $(\partial (X, {\mathcal E}), \rho_a)$  is totally disconnected.
\end{Cor}

By Theorem \ref {th3.6} we obtain the following Lipschitz equivalence on the self-similar sets.

\begin{theorem} \label{th3.9} Let $K$ and $K'$ be self-similar sets that are generated by two IFS's as in (\ref {eq2.2}) with the same number of similitudes and the same contraction ratio, and satisfy condition (H) in (\ref{eq2.4}). Assume the associated augmented trees are simple and the incidence matrices are  $(m, {\bf b})$-rearrangeable (in particular, primitive).
Then  $K$ and $K'$ are  Lipschitz equivalent, and are also Lipschitz equivalent to a dust-like self-similar set.
\end{theorem}

\begin{proof}
It follows from Theorem \ref{th3.6} that
\begin{equation} \label{eq3.3}
\partial(X, {\mathcal E}) \simeq \partial(X, {\mathcal E}_v) = \partial( Y,
{\mathcal E}_v )\simeq \partial(Y, {\mathcal E})
\end{equation}
(for the respective metrics $\rho_a$). Let $\varphi:
\partial(X, {\mathcal E})  \to \partial(Y, {\mathcal E})$
be the bi-Lipschitz map. With no confusion, we just denote these two
boundaries by $\partial X$, $\partial Y$ as before.

By Proposition \ref{th3.8}, there exist two bijections
$\Phi_1:\partial X\to K$ and $\Phi_2:\partial Y\to K'$ satisfying
(\ref{eq3.2}) with constants $C_1,C_2$, respectively. Define $\tau:
K \to K'$ as
$$
\tau = \Phi_2\circ \varphi\circ \Phi_1^{-1}.
$$
Then
$$
\begin {aligned}
|\tau(x) -\tau (y)| & \leq
C_2\ \rho_a(\varphi\circ\Phi_1^{-1}(x),\varphi\circ\Phi_1^{-1}(y))^\alpha\\
&\leq C_2C_0^\alpha\ \rho_a(\Phi_1^{-1}(x),\Phi_1^{-1}(y))^\alpha\\
&\leq
C_2C_0^\alpha C_1\ |x-y|.
\end{aligned}
$$
Let $C' = C_2C_0^\alpha C_1$, then
$$
|\tau (x) - \tau (y)| \leq C'|x -y|.
$$
Similarly, we have ${C'}^{-1} |x-y| \leq |\tau (x) - \tau (y)|$.
Therefore $\tau : K \to K'$ is a bi-Lipschitz map.

For the last statement, we can regard $(X, {\mathcal E}_v)$ as the augmented tree of an IFS that is strongly separated, and then apply the above conclusion.
\end{proof}

\begin{Cor}\label{cor3.10}
The IFS in  Theorem \ref{th3.9} satisfies the OSC.
\end{Cor}

\begin{proof} We  make use of the following well-known result of Schief \cite{Sc94} on a self-similar set $K$: let $s$ be the similarity dimension of $K$, then the IFS satisfies the OSC if and only if $0<{\mathcal H}^s(K) <\infty$.

Let $K$ be the self-similar set as in Theorem \ref{th3.9}, then it
is Lipschitz equivalent to a dust-like set $K''$.  It follows that
$0<{\mathcal H}^s(K'') <\infty$, so is $K$ by the Lipschitz
equivalence. Hence by Schief's criterion, the IFS for $K$ satisfies
the OSC.
\end{proof}

In Section 5, we will provide some interesting examples for the
Lipschitz equivalence of the totally disconnected self-similar sets in Theorem \ref {th3.9}.  We also remark that in Theorem \ref{th3.9} the condition on the augmented tree can be weaken, and the  proof still yields a very useful result.

\begin{Prop}\label{th3.10}
Let $K$ and $K'$ be self-similar sets that are generated by two
IFS's as in (\ref {eq2.2}) that  have the same number of
similitudes, same contraction ratio, and satisfy condition (H) in
(\ref{eq2.4}). Suppose the two IFS's satisfy  either (i) the OSC, or
(ii) the augmented trees are simple. Then
\begin{equation} \label{eq3.5}
K \simeq K'\ \Leftrightarrow \ \partial X  \simeq \partial Y.
\end{equation}
\end{Prop}

\begin{proof} The sufficiency of (\ref {eq3.5}) is always satisfied, as we can replace (\ref{eq3.3}) by the given condition $\partial X  \simeq \partial Y$, then follows from the same proof of Theorem \ref {th3.9}. The necessity follows by making use of the H\"older equivalence (\ref {eq3.2}) which is satisfied for cases (i) and (ii), and proceeds  with a similar estimation for $\varphi = \Phi_2^{-1} \circ \tau \circ \Phi_1$.
\end{proof}

We remark that the above theory of Lipschitz equivalence can also be
applied to study the self-affine systems. Let $B$ be a $d\times d$ expanding matrix (i.e., all the eigenvalues have moduli $>1$)
and let  $\{S_i\}_{i=1}^m$ with $S_i(x) = B^{-1} (x+d_i), d_i \in {\Bbb R}^d$ be the IFS. For the part of simple augmented tree, it is
clear that the notion can be defined and the hyperbolicity  in Proposition \ref{th3.4} follows by the same way. Moreover we have

\begin {Prop} \label {th3.11} For the IFS $\{S_i\}_{i=1}^m$ of self-affine maps as the above, Theorems \ref{th3.6} and  \ref{th3.6'} remain valid.
\end {Prop}

For the part involves the self-affine set on ${\Bbb R}^d$, we need
to use a device in \cite{HeLa08} by replacing the Euclidean norm
with an ``utlta-norm" adapted to the matrix $B$. By renorming, we
can assume without loss of generality that $||x|| \leq ||Bx||$. For
$0<\delta<1/2$, let $\varphi \geq 0 $ be a $C^\infty$ function
supported in the open ball $U_\delta$ centering at $0$ with $\varphi (x)
=\varphi (-x)$ and $\int_{{\Bbb R}^d} \varphi =1$. Let $V=
BU_1\setminus U_1$, and let $h = \chi_V\ast \varphi$ be the
convolution of the indicator function $\chi_V$ and  $\varphi$. Let
$q = |\det (B)|$ and define
$$
w(x) = \sum_{n=-\infty}^\infty q^{-n/q}h(B^nx)\qquad x \in {\Bbb
R}^d.
$$ Then $w(x)$ satisfies  (i) $w(x) = w(-x)$ and $w(x) =0$ if and
only if $x=0$, (ii) $w(Bx) = q^{1/d}w(x)$, and (iii) there exists
$\beta>1$ such that $w(x+y) \leq \beta \max \{w(x), w(y)\}$. This
$w$ is used as a distance (ultra-metric) to replace the Euclidean
distance to define the generalized Hausdorff measure ${\mathcal
H}_w^{\alpha}$,  Hausdorff dimension $\dim_H^w$, box dimension
$\dim_B^w$. Under this setting, most of the basic properties for the
self-similar sets (including Schief's basic result on OSC) can be
carried to the self-affine sets \cite {HeLa08}. To apply to here, we
need to adjust condition (H) (\ref{eq2.4}) to
\begin{equation*}
K_{\mathbf u}\cap K_{\mathbf v} = \emptyset \quad \Rightarrow \quad
\text{dist}_w(K_{\mathbf u},K_{\mathbf v})\geq cq^{-n/d}.
\end{equation*}
and  to replace the $r^n$ in the proofs of Lemma \ref{th2.4} and
Proposition \ref{th3.8} by $q^{-n/d}$. Then we have

\begin{theorem} \label {th3.12} With $K$ and $K'$ self-affine sets satisfying the conditions in Theorem \ref{th3.9}. The $K$ and $K'$ are Lipschitz equivalent under the ultra-metric defined by $w$, and they are Lipschitz equivalent to a dust-like self-affine set.
\end{theorem}
\end{section}

\bigskip

\begin{section}{\bf Rearrangeable matrix and proofs of the main theorems}

The proof of the Lipschitz equivalence of the simple augmented
tree to the original tree in Theorem \ref{th3.6} is to construct a near-isometry between them, which is based on a device of
``rearrangement" of graphs.  The idea of rearrangement  was introduced by Deng and He \cite{DeHe10}. A similar technique of ``equal decomposition" was also used to consider the Lipschitz equivalence in
\cite{RaRuXi06} (see also \cite{XiXi10}, \cite{XiXi11}).  First we give a detail discussion of the concept of rearrangement.

\begin{Def}\label{def4.1}
Given $m, r\in {\mathbb N}$. Suppose \ ${\mathbf
a}=[a_1,\dots,a_r]\in{\mathbb Z}_+^r$ is a row vector,  and
${\mathbf b}=[b_1,\dots,b_r]^t\in{\mathbb N}^r$ is a column vector.  We
say  that ${\mathbf a}$ is $(m,{\mathbf b})$-rearrangeable if there
exists an integer $p>0$, and a nonnegative integral matrix $C=[c_{ij}]_{p\times r}$ such that
$$
{\mathbf a} = [1, \dots , 1] C   \quad \hbox {and } \quad  C{\mathbf b} = [m, \dots,
m]^t.
$$
(Note that in this case ${\mathbf a}{\mathbf b}=pm $ for some $p\in{\mathbb N}$.)

 A matrix $A$ is called $(m,{\mathbf b})$-rearrangeable if  each row of $A$ is $(m,{\mathbf b})$-rearrangeable.
\end{Def}

\noindent {\bf Remarks.} (1)  The intuitive explanation of the
definition is as follows. Let $a_i$ be the number of balls with the
same weight $b_i$. That ${\bf a}$ is $(m, {\bf b})$-rearrangeable
means we can rearrange these balls into $p$ groups (the $p$-rows in
$C$) such that in each group the number of balls with weight $b_j$
is $c_{ij}$ and  the total weight is exactly $m$.  It it is clear
that the total weight of all balls is $pm = {\bf a}{\bf b}= {\bf 1}
C {\bf b}$.

\medskip

(2) It follows easily from the definition that if \ ${\bf a}$ \ is
$(m, {\bf b})$-rearrangeable, then \ $\max \{b_i: a_ib_i \not =0\}
\leq m$.

\medskip

(3) If we write an $r\times r$ matrix $A$ as  ${\tiny
\left[\begin{array}{c}
          {\mathbf a}_1  \\
          \vdots \\
          {\mathbf a}_r
    \end{array} \right]}$.
Then $A$ is $(m,{\mathbf b})$-rearrangeable is equivalent to the
existence of ${\mathbf p}=[p_1,\dots,p_r]^t\in{\mathbb N}^r$ such
that $A{\mathbf b}=m{\mathbf p}$ and a sequence of nonnegative integral matrices
$\{C_i\}_{i=1}^r$ such that ${\mathbf a}_i= {\mathbf 1}C_i$  and
$C_i{\mathbf b}=[m,\dots,m]^t\in{\mathbb N}^{p_i}$ for all $i$.

We will prove some sufficient conditions for $(m, {\bf b})$-rearrangeable in the sequel (Lemma \ref{th4.6}, Proposition \ref {th4.7}). In the following we use two examples to illustrate more on such notion.

\begin{Exa} Let ${\mathbf a} \in {\mathbb Z}_+^r$, and let ${\mathbf b} = {\mathbf 1}^t$. Suppose $\sum_ia_i =m$, then trivially,
$\mathbf a$ is $(m, {\mathbf b})$-rearrangeable (with $p=1$,  $C=
{\mathbf a}$, i.e.,  intuitively we put everything in one group).
Consequently, for any nonnegative integral matrix $A$ with $m$ as an
eigenvalue and \ ${\mathbf b}= {\mathbf 1}^t$ as the eigenvector, it
is $(m, {\mathbf b})$-rearrangeable.
\end{Exa}

\begin{Exa}\label{example of rearrangeable.}
Let
$$
A=\left[\begin{array}{rrr}
          1 &  1 & 0  \\
          1 &  1 &  1 \\
          1 & 1 &  2
    \end{array} \right]
$$
 and ${\mathbf b} = [1,2,3]^t$. Then $A {\mathbf b} = 3 {\mathbf b}$, and  $A$ is $(3,{\mathbf  b})$-rearrangeable, so is $A^2$.
\end{Exa}

\begin{proof}
To show that $A$ is $(3,{\mathbf  b})$-rearrangeable, it suffices to
check on each row of  $A$ is $(3,{\mathbf  b})$-rearrangeable:

for ${\mathbf a}=[1,1,0]$,  then ${\mathbf
a}{\mathbf b}=3$, we take $C={\mathbf a}$;

for ${\mathbf a}=[1,1,1]$,  then ${\mathbf
ab}=2\times 3$. we take $C={\tiny \left[\begin{array}{rrr}
          1 & 1 &0  \\
          0 & 0 &1 \\
    \end{array} \right]}$;

for ${\mathbf a}=[1,1,2]$,  then ${\mathbf
ab}=3\times 3$, we take $C={\tiny \left[\begin{array}{rrr}
          1 & 1 & 0 \\
          0 & 0 & 1\\
          0 & 0 & 1
    \end{array} \right]}$.

For $A^2={\tiny \left[\begin{array}{rrr}
          2 &  2 & 1  \\
          3 &  3 &  3 \\
          4 & 4 &  5
    \end{array} \right]}
$,  we can proceed in the same way to show that $A^2$ is also
$(3,{\mathbf b})$-rearrangeable. The corresponding matrices $C$
for the three rows of $A^2$ are the transposes of the following
matrices:  ${\tiny \left[\begin{array}{rrr}
          1 & 1 & 0 \\
          1 & 1 & 0\\
          0 & 0 & 1
    \end{array} \right],\quad
\left[\begin{array}{cccccc}
          1 & 1 & 0 & 1 & 0 & 0 \\
          1 & 1 & 0 & 1 & 0 & 0 \\
          0 & 0 & 1 & 0 & 1 & 1
    \end{array} \right]}$  and
${\tiny \left[\begin{array}{ccccccccc}
          1 & 1 & 0 & 1 & 0 & 0 & 1 & 0 & 0 \\
          1 & 1 & 0 & 1 & 0 & 0 & 1 & 0 & 0 \\
          0 & 0 & 1 & 0 & 1 & 1 & 0 & 1 & 1
    \end{array} \right]}$.
\end{proof}

As a matter of fact, the above example is typical by the following
proposition.

\begin{Prop} \label{th4.4}
Let ${\mathbf  b}=[b_1,\dots,b_r]^t\in{\mathbb N}^r$. If a matrix
$A=[a_{ij}]_{r\times r}$ is $(m,{\mathbf b})$-rearrangeable, then
$A^n$ is $(m,{\mathbf b})$-rearrangeable for  $n\geq
1$.

If in addition, the eigen-relation $A{\mathbf b}=m{\mathbf b}$
is satisfied, then $A^{n}$ is also $(m^{n},{\mathbf
b})$-rearrangeable.
\end{Prop}

\begin{proof} We  use induction to prove the first part.
Let ${\mathbf a}_i, \ 1\leq i\leq r$ be the row
vectors of $A$. Since $A$ is $(m,{\mathbf  b})$-rearrangeable,
there exist ${\mathbf p}=[p_1,\dots,p_r]^t\in{\mathbb N}^r$ such
that $A{\mathbf b}=m{\mathbf p}$ and a sequence of  matrices
$\{C_i\}_{i=1}^r$ such that ${\mathbf a}_i= {\mathbf 1}C_i$ and
$C_i{\mathbf b}=[m,\dots,m]^t\in{\mathbb N}^{p_i}$ for all $i$.

  Assume that $A^{n-1}$ is $(m, {\bf b})$-rearrangeable, then $A^{n-1} {\bf b} = m \widetilde {\bf p}$ for some positive integer vector $\widetilde{\bf p}$.
 Consider $A^n$, let ${\boldsymbol\alpha}_i$ be the $i$-th row of $A^n$. Since
$A^n{\mathbf b}=m A \widetilde{\mathbf p}$, we have ${\boldsymbol
\alpha}_i{\mathbf b}=m{\mathbf a}_i\widetilde{\mathbf p}: = mq_i$.
Write $A^{n-1} = [\widetilde a_{ij}]$,  it follows that
$$
{\boldsymbol \alpha}_i =\sum_{j=1}^r\widetilde a_{ij}{\mathbf a}_j =
\sum_{j=1}^r\widetilde a_{ij}{\mathbf 1}C_j .
$$
Let
$$
C^{(i)}=\big[\underbrace{C_1,\dots,C_1}_{\widetilde a_{i1}},\dots,
\underbrace{C_r,\dots,C_r}_{\widetilde a_{ir}}\big]^t
$$
where the transpose means transposing the row of matrices into a
column of matrices (without transposing the $C_j$ itself). Then
$$
{\boldsymbol\alpha}_i =  {\mathbf 1}C^{(i)} \quad \hbox {and} \quad
C^{(i)}{\mathbf b}=[m,\dots,m]^t\in {\mathbb
N}^{q_i}.
$$
Hence ${\boldsymbol \alpha}_i$ is $(m,{\mathbf b})$-rearrangeable, and  $A^n$ is $(m,{\mathbf b})$-rearrangeable.

For the second part, since $A{\mathbf b}=m{\mathbf b}$ and
$A^n{\mathbf b}=m^n{\mathbf b}$, we can replace the previous
integral vector ${\mathbf p}$ by ${\mathbf b}$. Then ${\boldsymbol
\alpha}_i{\mathbf b}=mq_i= m^np_i$. We then replace the $q_i \times
r $ matrix $C^{(i)}$ in the above by the $p_i\times r$ matrix
$D^{(i)}$ which is obtained by summing every consecutive $m^{n-1}$
row vectors of $C^{(i)}$ (note that $q_i=m^{n-1}p_i$). Hence
$$
{\boldsymbol \alpha}_i=  {\mathbf 1}D^{(i)} \quad \hbox {and}\quad
D^{(i)}{\mathbf b}=[m^n,\dots,m^n]^t\in {\mathbb N}^{p_i}.
$$
so that
${\boldsymbol \alpha}_i$ is $(m^n,{\mathbf b})$-rearrangeable.  Consequently, $A^n$ is $(m^n,{\mathbf b})$-rearrangeable.
\end{proof}

\noindent{\bf Proof of Theorem \ref {th3.6}}\quad  Recall that  $A{\bf b} = m{\bf b}$ where ${\mathbf b}=[b_1,\dots,b_r]^t$
with $b_i=\#T$ where $[T]={\mathscr T}_i$ (Proposition \ref {th3.5}). First we claim that we can assume without loss of generality that $\max_i b_i \leq m$. For otherwise, let  $k$ be sufficiently large such that $\max_i
b_i\leq m^{k}$,  by Proposition \ref{th4.4}, $A^{k}$ is
$(m^{k}, {\mathbf b})$-rearrangeable. The  IFS of the
$k$-th iteration of $\{S_i\}_{i=1}^m$  has symbolic
space  $X'={\bigcup}_{n=0}^{\infty}\Sigma^{k n}$ and the
augmented  tree has incidence matrix $A^{k}$; moreover  the two hyperbolic boundaries
$\partial X'$ and $\partial X$ are identical. Hence we can consider  $A^{k}$ instead if $\max_i b_i \leq m$ is not satisfied.

Let $X_1=(X,{\mathcal E}),~ X_2=(X,{\mathcal E}_v)$. In view of
Proposition \ref{th3.3},   it suffices to show that there exists a
near-isometry $\sigma$ between $X_1$ and $X_2$, and hence
$\partial(X, {\mathcal E} )\simeq
\partial(X, {\mathcal E}_v )$.  We define this $\sigma$ to be a
one-to-one mapping  from $\Sigma^n$ (in $X_1$) to $\Sigma^n$ (in
$X_2$) inductively as follows: Let
$$
\sigma(o)=o \quad \hbox {and} \quad \sigma(i)=i, ~i\in \Sigma.
$$
Suppose  $\sigma$ is defined on the level $n$ such that for
every horizontal connected component $T$, $\sigma (T)$ has the same
parent (see Figure 1), i.e.,
\begin{equation} \label {eq4.2}
\sigma (x)^{-1} = \sigma (y)^{-1} \qquad  \forall \ x, y \in T
\subset \Sigma^n
\end{equation}
To define the map $\sigma$ on  $\Sigma^{n+1}$, we note that  $T$ in
$\Sigma^n$ gives rise  to horizontal connected components in
$\Sigma^{n+1}$, which are accounted by the incidence matrix $A$. We can
write
$$
T\Sigma= {\bigcup}_{k=1}^\ell Z_k.
$$
where $Z_k$ are horizontal connected components consisting of
offsprings of $T$. If $T$ belongs to the connected class  ${\mathscr
T}_i$, then $\#T=b_i$.  By the definition of the incidence matrix
$A$ and $A{\bf b} = m{\bf b}$, we have
$$
b_im\ =\ {\sum}_{k=1}^\ell\#{Z_k}\ = \ {\sum}_{j=1}^ra_{ij}b_j.
$$
Since $A$ is $(m,{\mathbf  b})$-rearrangeable, for the ${\mathbf
a}_i$, there exists a nonnegative integral matrix $
C=[c_{sj}]_{b_i\times r}$ (depends on $i$) such that
$$
{\mathbf a}_i = {\mathbf 1} C \quad \hbox {and} \quad C{\mathbf  b}
=[m,\dots,m]^t.
 $$
We decompose ${\mathbf a}_i$ into $b_i$ groups according to $C$ as
follows. Note that $a_{ij}$ denotes the number of $Z_k$ that belongs
to ${\mathscr T}_j$. For each $1\leq s \leq b_i$ and  $1\leq j\leq
r$, we choose  $c_{sj}$ of those $Z_k$, and denote by $\Lambda_s$
the set for all the chosen $k$ with $1\leq j\leq r$. Then we can
write the index set $\{1,2,\dots,\ell\}$ as a disjoint union:
$$
\{1,2,\dots,\ell\}={\bigcup}_{s=1}^{b_i}{\Lambda_s}.
$$
Hence $\bigcup_{k=1}^{\ell}Z_{k}$ can be rearranged as $b_i$ groups
so that the total size of every group is equal to $m$, namely,
\begin{equation} \label {eq4.3}
{\bigcup}_{k=1}^{\ell}Z_k={\bigcup}_{k\in \Lambda_1}Z_{k}\ \cup \
\cdots \cup \ {\bigcup}_{k\in \Lambda_{b_i}}Z_{k}.
\end{equation}
Note that each set on the right has $m$ elements.
\begin{figure}[h]
  \centering
  \includegraphics[width=12cm]{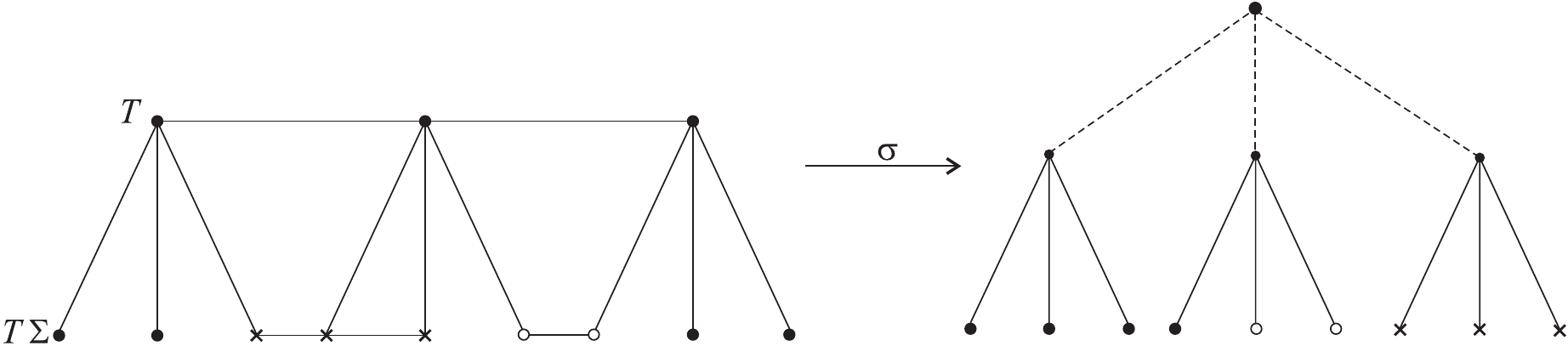}
\caption{ An illustration of the rearrangement by $\sigma$, the
{\tiny $\bullet , \ \circ$} and {\tiny $\times$} on the left denote
the types of connected components}\label{fig.re}
\end{figure}

For the  connected component  $T = \{{\mathbf i}_1,  \dots ,
{\mathbf i}_{b_i}\} \subset \Sigma^n$, we have defined $\sigma$ on
$\Sigma^n$ and $ \sigma (T) = \{{\mathbf j}_1 = \sigma ({\mathbf
i}_1), \ \dots ,\ {\mathbf j}_{b_i} = \sigma ({\mathbf i}_{b_i})\}.
$ In view of (\ref {eq4.3}), we define $\sigma$ on $T\Sigma =
{\bigcup}_{k=1}^{\ell}Z_k$ by assigning each ${\bigcup}_{k\in
\Lambda_s}Z_{k}$ (it has $m$ elements) the $m$ descendants of
${\mathbf j}_s$ (see Figure \ref{fig.re}). It is clear that $\sigma$
is well-defined on $T\Sigma$ and satisfies (\ref{eq4.2}) for $x,y
\in T\Sigma$. We apply the same construction of $\sigma$ on the
offsprings of every horizontal connected component in $\Sigma^n$. It
follows that $\sigma$ is well-defined  and satisfies (\ref{eq4.2})
on $\Sigma^{n+1}$. Inductively, $\sigma$ can be defined from $X_1$
to $X_2$ and  is bijective.

Finally we show that  $\sigma$ is indeed a near-isometry and
complete the proof.  Since $\sigma: X_1 \to X_2$ preserves the
levels,  without loss of generality, it suffices to prove the
near-isometry for ${\mathbf x,y}$ belong to the same level. Let
$\pi({\mathbf x,y})$  be the canonical geodesic connecting them,
which can be written as
$$
\pi({\mathbf x,y})=[{\mathbf x}, {\mathbf u}_1,\dots, {\mathbf u}_n,
{\mathbf t}_1,\dots, {\mathbf t}_k, {\mathbf v}_n,\dots, {\mathbf
v}_1, {\mathbf y}]
$$
where $[{\mathbf t}_1,\dots, {\mathbf t}_k]$ is the horizontal part
and $[{\mathbf x}, {\mathbf u}_1,\dots, {\mathbf u}_n, {\mathbf
t}_1],~ [{\mathbf t}_k, {\mathbf v}_n,\dots, {\mathbf v}_1, {\mathbf
y}]$ are vertical parts. Clearly, $\{{\mathbf t}_1,\dots, {\mathbf
t}_k\}$ must be included in one horizontal connected component of
$X_1$, we denote it by $T_j$. With the notation as in Theorem
\ref{th2.3}(i), it follows  that for ${\mathbf x}\ne {\mathbf y}\in
X_1$,
$$
|\pi({\mathbf x,y})| =|{\mathbf x}|+|{\mathbf y}|-2l+h,  \qquad |\pi
(\sigma({\mathbf x}),\sigma({\mathbf y}))|=|\sigma({\mathbf
x})|+|\sigma({\mathbf y})|-2l^{\prime}+h^{\prime}.
$$
We have $$\big ||\pi (\sigma({\mathbf x}),\sigma({\mathbf
y}))|-|\pi({\mathbf x,y})|\big |\leq |h-h'|+2|l'-l|\leq k+2|l'-l|$$
where $k$ is a hyperbolic constant as in Theorem \ref{th2.3}(ii). If
$T_j$ is a singleton, then
$$|l'-l|=0.$$
If  $T_j$ contains more than one point, then the elements of
$\sigma(T_j)$ share the same parent. Then the confluence of
$\sigma({\mathbf x})$ and $\sigma({\mathbf y})$ (as a tree) is
$\sigma ({\mathbf x)}^{-1} \ ( = \sigma ({\mathbf y)}^{-1})$. Hence
$$|l'-l|=1.$$
Consequently
$$
\big ||\pi(\sigma({\mathbf x}),\sigma({\mathbf y}))|-|\pi({\mathbf
x, y})|\big |\leq k+2.
$$
This completes the proof that $\sigma$ is a near-isometry and the
theorem is established.  \hfill  $\Box$

\begin{Cor} \label{cor4.6}
Under the same assumption of Theorem \ref{th3.6}. If there exists
$k\in {\mathbb N}$ such that $A^k$ is $(m^k,{\mathbf
b})$-rearrangeable, then  $\partial (X,{\mathcal E})\simeq \partial (X, {\mathcal E}_v)$.
\end{Cor}

To  prove of  Theorem \ref{th3.6'} that primitive implies $(m, {\bf b})$-rearrangeable,   we need a
combinatorial lemma due to Xi and Xiong \cite{XiXi10}. We include their proof for completeness.

\begin{Lem} \label{th4.5}
Let $p,m$ and $\{n_i\}_{i\in \Lambda}$ be positive integers with
$\sum_{i\in \Lambda}n_i=pm$. Suppose there exists an integer $l$
with $n_i\leq l < m$ for all $i\in \Lambda$,  and $\#\{i\in \Lambda:
n_i=1\}\geq pl$. Then there is a decomposition
$\Lambda=\bigcup_{s=1}^p\Lambda_s$ satisfying $\sum_{i\in
\Lambda_s}n_i=m$ for $1\leq s\leq p$.
\end{Lem}

\begin{proof}
The lemma is trivially true for $p=1$, suppose it is true for $p$.
For $p+1$, let $\Omega\subset\{i: n_i=1\}$ with $\#\Omega=(p+1)l$,
and select a maximal subset $\Delta_1$ of $\Lambda\setminus \Omega$
such that $\sum_{i\in \Delta_1}n_i<m$. We claim that
$$
{\sum}_{i\in \Delta_1}n_i\ \geq\  m-l.
$$
For otherwise, take $i_0\in \Lambda\setminus(\Delta_1\cup\Omega)$,
then $\sum_{i\in \Delta_1\cup\{i_0\}}n_i<(m-l)+l=m$, which
contradicts the maximality of $\Delta_1$ and the claim follows.
Choose a subset $\Omega_1$ from $\Omega$ such that
$\#\Omega_1=m-\sum_{i\in \Delta_1}n_i ~(\leq l)$ and set
$\Lambda_1=\Delta_1\cup \Omega_1$, then $\sum_{i\in
\Lambda_1}n_i=m$.

Note that for $\Lambda'=\Lambda\setminus\Lambda_1$, $\#\{i\in
\Lambda': n_i=1\}\geq pl$. Applying the inductive hypothesis on $p$,
we get a decomposition $\Lambda'=\bigcup_{s=2}^{p+1}\Lambda_s$ with
$\sum_{i\in \Lambda_s}n_i=m$ for $s\geq 2$. Therefore, the assertion
for $p+1$ holds, and the lemma is proved.
\end{proof}

Lemma \ref{th4.5} yields the following rearrangement lemma we need (see also \cite{DeHe10}).

\begin{Lem}\label{th4.6}
Let ${\mathbf  b}=[b_1,\dots,b_r]^t\in {\mathbb N}^r$ with $b_1=1$.
Let $\ell = {\max_j b_j}$ and ${\mathbf a}=[a_1,\dots,a_r]\in
{\mathbb Z}_+^r$ such that $a_1\geq \ell^2$. Suppose there exists $m
> \ell$
 such that ${\mathbf ab}=pm$ for some integer $0<p\leq \ell$,  then ${\mathbf a}$ is $(m,{\mathbf b})$-rearrangeable.
\end{Lem}

\begin{proof} We first assume that all $a_i >0$.
By the assumption, we have $\ell < m$ and $a_1\geq \ell^2\geq p
\ell$. Define a sequence $\{n_j\}_{j=1}^r$ by
\begin{equation*}
n_j  = \left \{
\begin {array} {ll}
 b_1  \ \  &j=1,\dots,a_1;\\
  b_2  \ \  &j=a_1+1,\ \dots,\ a_1+a_2;\\
 \vdots   &\\
 b_r  \ \ &j=\sum_{j=1}^{r-1}a_j+1,\
\dots,\ \sum_{j=1}^r a_j.
\end{array} \right .
\end{equation*}
Note that $n_j \leq \ell$. Let $\Lambda=\{1,2,\ \dots, \
\sum_{j=1}^r a_j\}$ be the index set. Then the assumption ${\bf a}{\bf b} =pm$  is
equivalent to:
$$
{\sum}_{j\in\Lambda }n_j = pm.
$$
By Lemma \ref{th4.5}, there is a decomposition
$\Lambda=\bigcup_{s=1}^p\Lambda_s$ satisfying $\sum_{j\in
\Lambda_s}n_j=m$ for $1\leq s\leq p$. Counting the number of
$b_j$'s in each group $s$,
$$
c_{sj}=\#\{k\in \Lambda_s: n_k=b_j\}, \quad 1\leq s\leq p, \ 1\leq
j\leq r.
$$
The lemma follows by letting the matrix $C=[c_{sj}]_{p\times r}$.

If some of the $a_i$ equals zero. Without loss generality we assume
that $a_r=0$ and  $a_i >0,\ 1\leq i \leq r-1$. Let ${\bf a}' = [a_1,
\dots , a_{r-1}]$ and ${\bf b}' =  [b_1, \dots b_{r-1}]$, then ${\bf
a}'{\bf b}' = pm$ and by the above conclusion, ${\bf a}'$ is $(m,
{\bf b})$-rearrangeable by a $p\times (r-1)$ matrix $C'$.  Let $C$
be the $p\times r$ matrix obtained by adding a last column ${\bf 0}$
to $C'$. Then it is easy to see that ${\bf a}$ is $(m, {\bf
b})$-rearrangeable.
\end{proof}

\begin{Prop}\label{th4.7}
Let $A$ be an $r\times r$ nonnegative integral matrix  and is
primitive (i.e., $A^n>0$ for some $n>0$). Let ${\mathbf
b}=[b_1,\dots,b_r]^t\in {\mathbb N}^r$ with $b_1=1$ and $A{\mathbf
b}=m{\mathbf b}$ for some $m\in{\mathbb N}$. Then $A^k$ is
$(m^k,{\mathbf b})$-rearrangeable for some $k>0$.
\end{Prop}

\begin{proof}
 Let $\ell = {\max_j b_j}$.  Observe that $A$ is primitive, there exists a large  $k$ such that $\ell < m^k$,  and
$A^k>0$ with every entry greater than $\ell^2$. Hence Lemma
\ref{th4.6} implies $A^k$ is $(m^k, {\mathbf b})$-rearrangeable.
\end{proof}

\noindent{\bf Proof of Theorem \ref{th3.6'}} \quad  Let $A$ be the incidence matrix, and $A{\bf b} = m{\bf b}$ where ${\mathbf b}=[b_1,\dots,b_r]^t$
with $b_i=\#T$ where $[T]={\mathscr T}_i$ (Proposition \ref
{th3.5}). If we let the root $o$ to be ${\mathscr T}_1$, then it is
clear from the definitions of $A$, $\bf b$ and $A{\bf b} = m {\bf
b}$ that $b_1 =1$. Hence Proposition \ref {th4.7} implies that $A^k$
is $(m^k, {\bf b})$-rearrangeable for some $k>0$.

 To prove the last part, we can assume without loss of generality that
{\it $A$ is $(m, {\mathbf b})$-rearrangeable and
$\max_i b_i \leq m$}.
For otherwise, by the primitive assumption on $A$ and by the same reasoning as in the proof of Theorem \ref{th3.6},
we can consider the augmented tree as the IFS defined by the $k$-th iteration of $\{S_i\}_{i=1}^m$,  and the corresponding $A^k$ is
$(m^k, {\mathbf b})$-rearrangeable. Hence we can apply  Theorem \ref{th3.6} and complete the proof of Theorem \ref{th3.6'}.   \hfill $\Box$

\end{section}

\bigskip

\begin{section}{\bf Examples}

In this section, we provide several examples to illustrate our
theorems of simple augmented tree and Lipschitz equivalence.  Note that all the IFS's considered here satisfy condition (H); also the bounded closed invariant sets $J$ we use are connected tiles and
they have finite number of neighbors. By using the following  lemma, we see that the process of finding the connected components will end in finitely many steps.

We assume the IFS consists of contractive similutudes $S_i(x)= B^{-1}(x+d_i), i=1,\dots, m$,
where $B^{-1}= rR$ is the scaled orthogonal matrix.  Let $J$ be a closed subset such that $S_i(J) \subset J$ for each $i$. Then for ${\bf
i}=i_1\cdots i_k\in \Sigma^k$, we have  $J_{\bf i}=S_{\bf
i}(J)=B^{-k}(J+d_{\bf i})$ where $d_{\bf i}= d_{i_k}+
Bd_{i_{k-1}}+\cdots+ B^{k-1}d_{i_1}$.

\begin{Lem}\label{th5.0}
For two horizontal connected components $T_1\subset \Sigma^{k_1}, T_2\subset
\Sigma^{k_2}$ with $\#T_1 = \#T_2=\ell$. If there exist
smilitudes $\phi_i(x)= B^{k_i}x+c_i,\ i=1,2$, where $c_i\in
{\Bbb R}^d$ such that
$$
\phi_1(\cup_{{\bf i}\in T_1}J_{\bf i})=\phi_2(\cup_{{\bf j}\in
T_2}J_{\bf j})=J\cup(J+v_1)\cup\cdots\cup(J+v_{\ell-1})
$$
for some vectors $v_1,\dots, v_{\ell-1}\in {\Bbb R}^d$, then $T_1\sim T_2$.
\end{Lem}

\begin{proof}
By the definition and (\ref{eq3.0}), it suffices to prove
$T_1\Sigma$ and $T_2\Sigma$ have the same connectedness structure,
equivalently, to prove this for $\bigcup_{{\bf i}\in T_1}\bigcup_{i=1}^m
J_{{\bf i}i}$ and $\bigcup_{{\bf j}\in T_2}\bigcup_{j=1}^m J_{{\bf
j}j}$. Letting $v_0=0$, we have
\begin{eqnarray*}
 \phi_1\big(\bigcup_{{\bf i}\in
T_1}\bigcup_{i=1}^m J_{{\bf
i}i}\big)
 &=&  B^{k_1}\big(\bigcup_{{\bf i}\in T_1}\bigcup_{i=1}^m J_{{\bf i}i}\big)+c_1 \\
 &=&  B^{k_1}\big(\bigcup_{{\bf i}\in T_1}\bigcup_{i=1}^m B^{-k_1-1}(J+d_i+Bd_{\bf i})\big)+c_1 \\
 &=&  \bigcup_{{\bf i}\in T_1}\bigcup_{i=1}^m (J_i+d_{\bf
 i})+c_1
 =  \bigcup_{j=0}^{\ell-1}\bigcup_{i=1}^m (J_i+v_j).
\end{eqnarray*}
Similarly we have $\phi_2\big(\bigcup_{{\bf j}\in T_2}\bigcup_{j=1}^m
J_{{\bf j}j}\big)= \bigcup_{j=0}^{\ell-1}\bigcup_{i=1}^m
(J_i+v_j).$  Since $\phi_1,\phi_2$ are similarity maps which preserve the
connectedness,  the lemma follows.
\end{proof}

We remark that it is a lot simpler to use the approach here than
the method of constructing the graph directed system as in
\cite{DeHe10}, \cite{RaRuXi06}, \cite{XiXi10}. In fact Example
\ref{add.ex} shows that it may be difficult to use the later to
prove the Lipschitz equivalence. The same example also shows that
the converse of the above lemma does not hold.

\begin{Exa}\label{example1}{
 Let ${\mathcal D}_1=\{0,2,4\}$ and ${\mathcal D}_2=\{0,3,4\}$ be two
digit sets, and let  $K_1$ and $K_2$ be the two self-similar sets
defined by:
$$
K_1 =\frac{1}{5}(K_1+{\mathcal D}_1), \quad K_2
=\frac{1}{5}(K_2+{\mathcal D}_2).
$$
Then $K_1$ is clearly dust-like and $K_1 \simeq K_2$ (see Figure
\ref{fig.1}).}
\end{Exa}

  This is a well-known example of Lipschitz equivalence of self-similar sets called the {\it $\{1,3,5\}-\{1, 4, 5\}$ problem} (\cite{DaSe97}, \cite{RaRuXi06}),
 which was proved by showing that they have the same graph directed system. In our approach, the associated augmented tree for
 $K_1$ is just the standard rooted tree with no horizontal edges, and the incidence matrix is $[3]$. For  $K_2$, by letting $S_i (x) = (x+ d_i)/3,\ d_i \in {\mathcal D}_2$, and use $J = [0,1]$ in the definition of the augmented tree,
  we see that in total, there are three connected classes (after three iterations): ${\mathscr T}_1=[{o}],
{\mathscr T}_2=[\{2,3\}]$ and ${\mathscr T}_3=[\{22, 23, 31\}]$. Hence
$X$ is simple and the incidence matrix is
$$
A=\left[\begin{array}{rrr}
          1 &  1 & 0  \\
          1 &  1 &  1 \\
          1 & 1 &  2
    \end{array} \right].
$$
Clearly $A$ is primitive, and  $K_2 \simeq K_1$  by Theorem
\ref{th3.9}. \hfill $\Box$

\medskip

\begin{figure}[h]
  \centering
  \subfigure[]{
  \includegraphics[width=5cm]{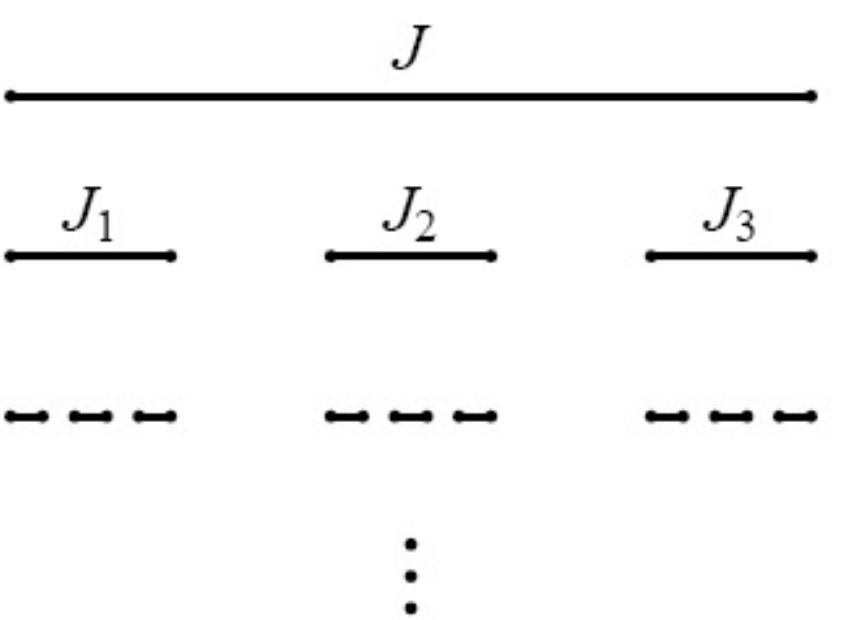}
 }
 \qquad
 \subfigure[]{
  \includegraphics[width=5cm]{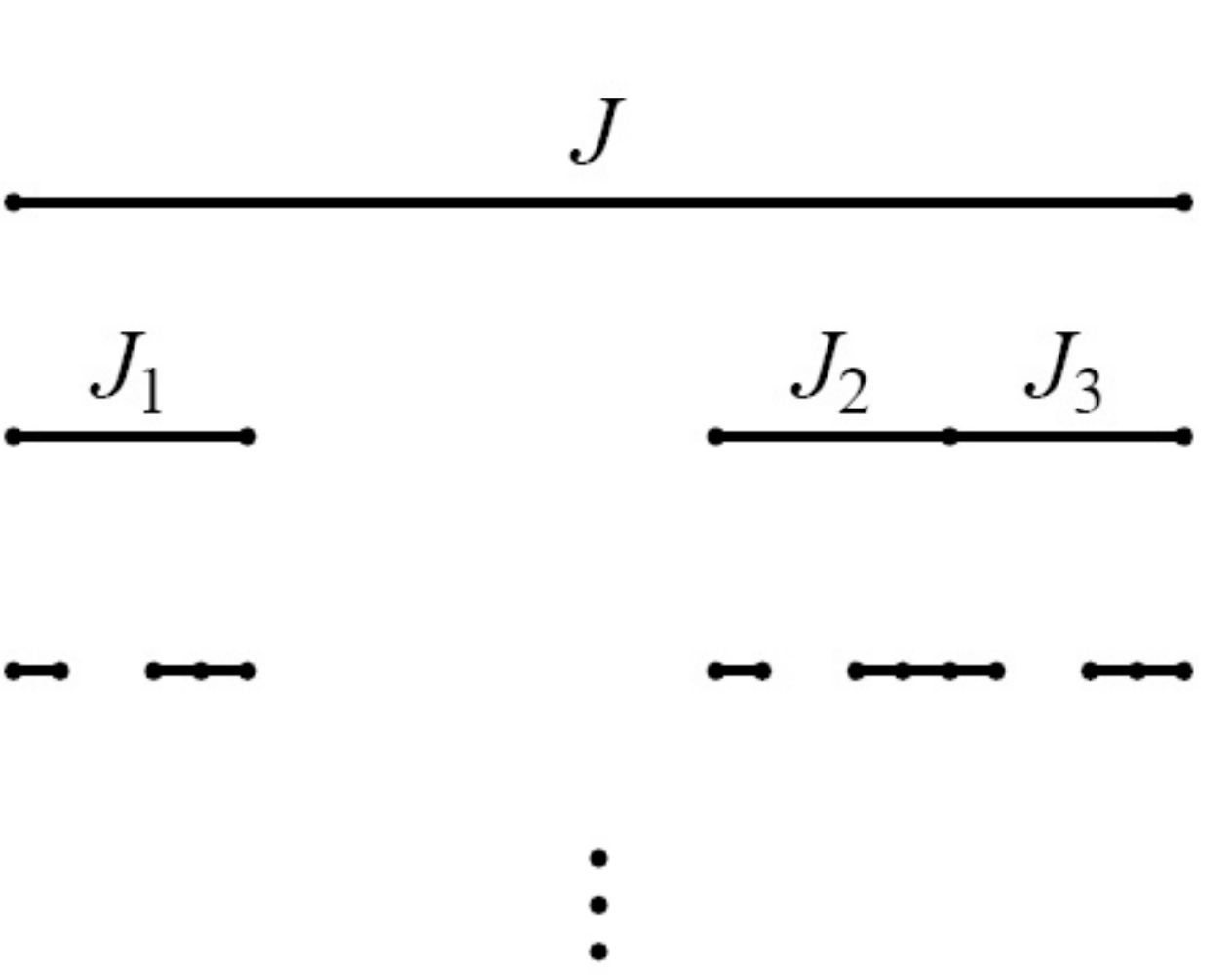}
 }
\caption{}\label{fig.1}
\end{figure}

The following proposition is a generalization of Example \ref{example1} on ${\Bbb R}$ \cite{RaRuXi06}.

\begin{Prop} \label{th5.2}
Let ${\mathcal D}$ be a proper subset of $\{0,1,\dots,n-1\}$, and
let
 $K$ the self-similar set on ${\Bbb R}$ satisfying
$$
 K =\frac{1}{n}(K+{\mathcal D}).
$$
Then the corresponding augmented tree is simple and the
incidence matrix is primitive. Consequently  $K$ is Lipschitz equivalent to a dust-like set.
\end{Prop}

\begin{proof}   Let $S_j(x)=\frac{1}{n}(x+d_j),\ d_j\in {\mathcal D}$ be the IFS for $K$ with  $\# {\mathcal D}= m (<n)$. We will ignore
the trivial case that the IFS is strongly separated. We use  $J=[0,1]$ to define the associated augmented tree $X$, and
it is  easy to see that it is simple.  Our main task is to show that the incidence matrix is primitive.
 Let $F_k = \bigcup_{{\bf i}\in \Sigma^k} S_{\bf i}(J)$ be the $k$-th iteration. Let $L_i, 1\leq i\leq \ell$ denote the
 disjoint closed sub-intervals of  $F_1$ arranged from  left to right.

We first consider the case that ${\mathcal D}$ does not contain both
$0$ and $n-1$. In this case,  the connected components of the
augmented tree is determined by the root and the components in the
first level. Let $o$ be the root which represents the connected
class ${\mathscr T}_1$. For the others, let $b_i\ (=n |L_i|)$ be the
number of vertices in the connected component where $ i=1,\dots, r$
for some $r\leq \ell$.  For convenience, we assume $b_1=1$ and $b_r$
corresponds the rightmost sub-interval of $F_1$.  Note that $b_i
>1$ for $2\leq i \leq r$.  Let $a_i$ denote the number of such
components. Then it is easy to see that the incidence matrix is
given by
$$
A=\left[\begin{array}{llll}
          a_1 &  a_2 & \cdots  & a_r   \\
          a_1b_2 &  a_2b_2 & \cdots & a_rb_2 \\
          \vdots  & \vdots &  & \vdots \\
          a_1b_r &  a_2b_r & \cdots & a_{r}b_r
          \end{array}\right].
$$
Hence $A$ is primitive and the proposition follows.

For the case that $\{0, n-1\} \subset {\mathcal D}$, we note that
for an interval $L_i$ in $F_1$ with $ b_i >1$, if we inspect the
sub-intervals of $F_2$ contained in $L_i$, we find that there is a
group of new sub-intervals that come from union of two intervals: one
with right end points $S_{im} (1)$ and one with left endpoint the
$S_{{i+1},1}(0)$ (see Figure \ref{fig.1}(b),  the forth interval on
the third level). There are $b_i-1$ of them and the length is  $1 +
b_r$. If there is already a component of this length, say the second
one, then the incidence matrix is an $r\times r$ matrix
$$
A={\small \left[\begin{array}{lllll}
          a_1 &  a_2 & \cdots & a_{r-1} & a_{r} \\
          a_1b_2-(b_2-1) &  a_2b_2+ (b_2 -1) & \cdots & a_{r-1}b_2 & a_{r}b_2-(b_2-1) \\
          \vdots  & \vdots  &  & \vdots & \vdots \\
          a_1b_r-(b_r-1) &  a_2b_r +(b_r-1) & \cdots & a_{r-1}b_r & a_{r}b_r-(b_r-1)
          \end{array}\right].}
$$
If this is a new component, then the incidence matrix is an $(r+1)\times (r+1)$ matrix with
$$
A={\small \left[\begin{array}{llllll}
          a_1 &  a_2 & \cdots & a_{r-1} & a_{r} & 0  \\
          a_1b_2-(b_2-1) &  a_2b_2 & \cdots & a_{r-1}b_2 & a_{r}b_2-(b_2-1) & b_2-1 \\
          \vdots  & \vdots  &  & \vdots & \vdots & \vdots \\
          a_1b_r-(b_r-1) &  a_2b_r & \cdots & a_{r-1}b_r & a_{r}b_r-(b_r-1) & b_r-1\\
          a_1(1+b_r)-b_r &  a_2(1+ b_r) & \cdots & a_{r-1}(1+b_r) & a_{r}(1+b_r)-b_r & b_r
          \end{array}\right].}
$$
Again, $A$ is primitive that implies the proposition.
\end{proof}

The above proposition also has a higher dimensional extension to the
totally disconnected {\it fractal cubes}  of the form $K =
\frac{1}{n}(K + {\mathcal D})$ where ${\mathcal D} \subset \{0,
\dots , n-1\}^d$  by Xi and Xiong \cite{XiXi10} via constructing a
more complicated graph directed system. Their result can also be put
into the framework of Theorem \ref{th3.6'},  we will omit the
detail. In the following we will consider other totally disconnected
self-similar sets of more general forms.

 The following shows that the $S_j(J)$'s can have large overlaps (though the OSC is still satisfied by Corollary
\ref{cor3.10}). Moreover it shows that the  converse of Lemma \ref{th5.0} does not hold.

\begin{Exa} \label{add.ex}
{
On ${\Bbb R}^2$, let  $K$  be the self-similar set defined by the
IFS
$$
 S_i(x)=\frac{1}{5}(x+d_i) \quad \hbox {where} \quad
 d_i\in  \left\{\left[\begin{array}{rr} 0 \\
0\end{array}\right],\left[\begin{array}{rr} 18/5 \\
0\end{array}\right], \left[\begin{array}{rr} 4 \\
0\end{array}\right], \left[\begin{array}{rr} 0 \\
2\end{array}\right], \left[\begin{array}{rr} 0 \\
3\end{array}\right]\right\}.
$$
We use $J=[0,1]^2$ as the invariant set to define an augmented tree
$X$ for $K$. It can be shown that, in the graph $X$, the two components
$T_1=\{2,3\}$ and $T_2=\{4,5\}$ are equivalent in the augmented tree, as the $T_1\cup T_1\Sigma$ and $T_2\cup T_2\Sigma$ are graph isomorphic
(see Figure \ref{additional ex}), but there do not exist
similarity maps $\phi_1, \phi_2$ as in Lemma \ref{th5.0} such that
$\phi_1(J_2\cup J_3)= \phi_2(J_4\cup J_5)$.

On the other hand, if we continue the process by iterating, we can
see that $X$ is simple with two connected classes ${\mathscr
T}_1=[o]$ and ${\mathscr T}_2=[T_1]$, and the incidence matrix is
 {\small
$\left[\begin{array}{rr}
           1 & 2 \\
           2 & 4
    \end{array} \right]$}  which is primitive. Hence $K$ is
    Lipschitz equivalent to a dust-like self-similar set by Theorem
\ref{th3.9}.

\medskip

\begin{figure}[h]
  \centering
  \subfigure[]{
  \includegraphics[width=5cm]{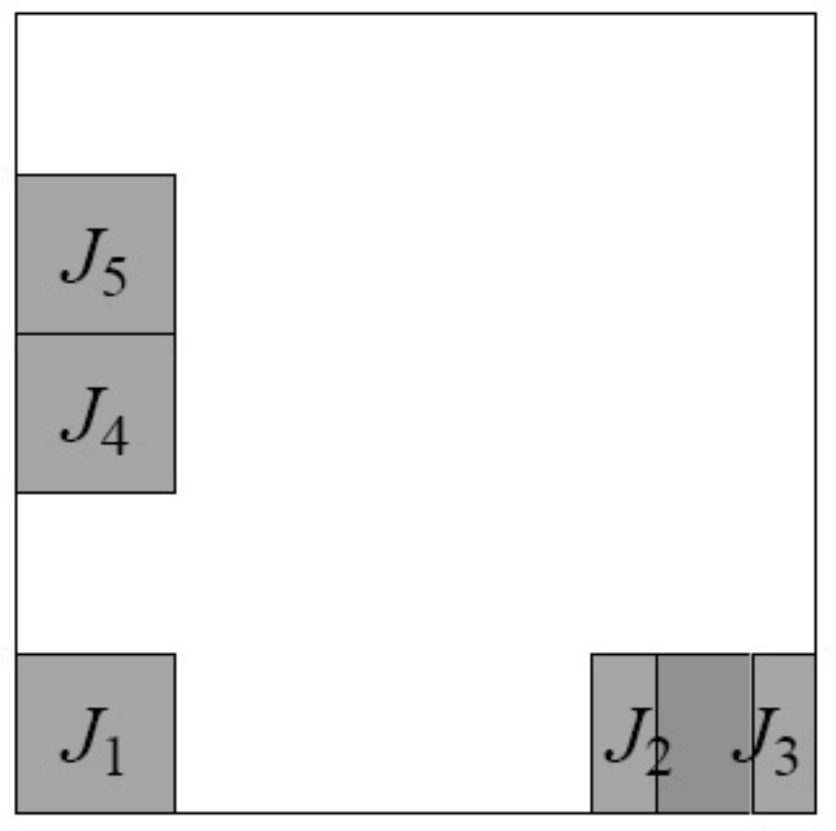}
 }\qquad
 \subfigure[]{
  \includegraphics[width=5cm]{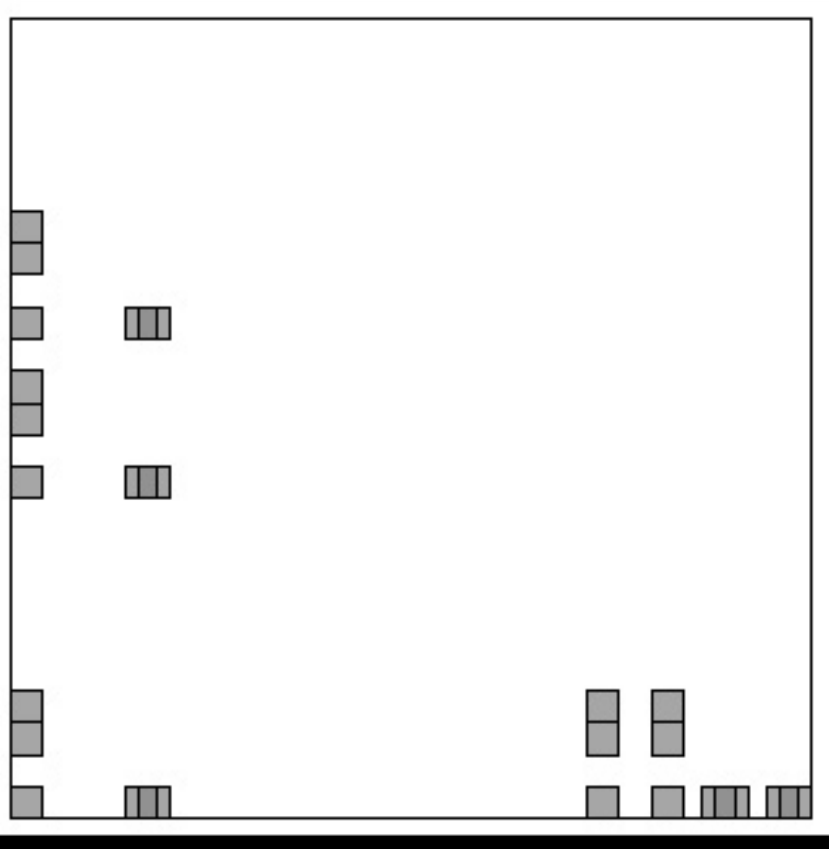}
 }
\caption{The first two iterations for $K$.}\label{additional ex}
\end{figure}
}
\end{Exa}

In the following we consider the self-similar set selected from a
self-similar tile (also the more general self-affine case) to
replace the square and sub-squares in the previous case. Let $B$ be
a $d\times d$  expanding matrix (i.e., all the eigenvalues have
moduli $>1$) and $|\det(B)|=q$; let ${\mathcal
D}=\{d_1,\dots,d_q\}\subset {\Bbb R}^d$ be a digit set. The
attractor $J$ generated by
$$
S_i (x) = B^{-1}(x + d_i), \qquad  d_i\in {\mathcal D}
$$
satisfies $J =B^{-1}( J+ {\mathcal D})$ (more conveniently, we use
$BJ =J+ {\mathcal D}$). It is called a {\it self-affine tile} if the
interior of $J$ is nonempty \cite{LaWa96}. If in addition, $B$ is a
similar matrix, then $J$ is a {\it self-similar tile}.  It is well-known
that if the tile $J$ is disk-like (i.e., homeomorphic to the unit
disc), then $J$ has either six (hexagonal) or eight (square)
neighbors \cite{BaWa01}. This is needed to determine the connected
classes of the augmented tree of the underlying self-similar set.

\begin{Exa} \label{ex5.6}{
$ B=\left[\begin{array}{rr}
           5/2         & -\sqrt{3}/2  \\
          \sqrt{3}/2  &  5/2
    \end{array} \right]
\quad \hbox {and} \quad {\mathcal D}=   \{0,\ \pm v_1,\ \pm v_2,\
\pm(v_1+v_2) \}.
$
with $v_1= \tiny {\left[\begin{array}{rr} 1/2 \\
\sqrt{3}/2\end{array}\right]} ,\  v_2= \tiny {\left[\begin{array}{rr} 1/2 \\
-\sqrt{3}/2\end{array}\right]}$.
If we let $S_i(x) = B^{-1}(x+d_i), d_i \in {\mathcal D}$, then $J$ is a self-similar tile on ${\Bbb R}^2$, and is called the {\it Gosper island} (Figure \ref{example5}(a)). If we let
\begin{equation*}
{\mathcal D}'=\{v_2,\ v_1,\ -v_2,\ -(v_1+v_2)\}
\end{equation*}
and let $K$  be the corresponding self-similar set (see Figure
\ref{example5}(b)). We use $J$ to define the augmented tree.   There
are six neighbors of $J$, they are
\begin{equation} \label{eq5.0}
J \pm v_1,\ \  J \pm v_2,\ \ J \pm (v_1+v_2).
\end{equation}
By making use of (\ref{eq5.0}), we can show that the incidence
matrix is
$$
A=\left[\begin{array}{rrrrrr}
         1 & 1 & 0 & 0 & 0 & 0  \\
         2 & 0 & 1 & 1 & 0 & 0  \\
         2 & 1 & 1 & 0 & 1 & 0  \\
         4 & 0 & 1 & 1 & 0 & 1  \\
         4 & 1 & 1 & 0 & 1 & 1  \\
         6 & 1 & 1 & 0 & 1 & 2
    \end{array} \right]$$
and it is primitive. Hence $K$ is Lipschitz equivalent to a
dust-like set. We omit the detail.}
\end{Exa}

\medskip

\begin{figure}[h]
  \centering
  \subfigure[]{
  \includegraphics[width=5cm]{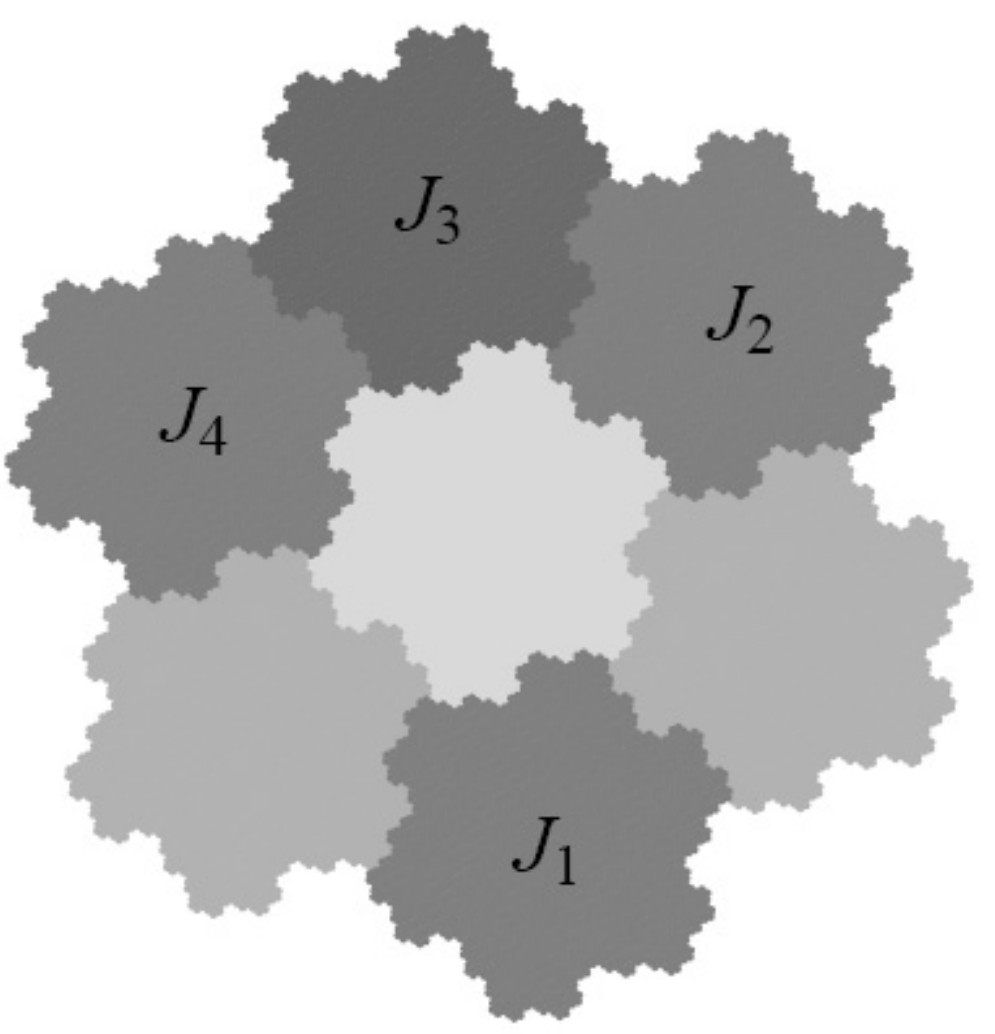}
 }
 \qquad
 \subfigure[]{
  \includegraphics[height=5cm, width=4.5cm]{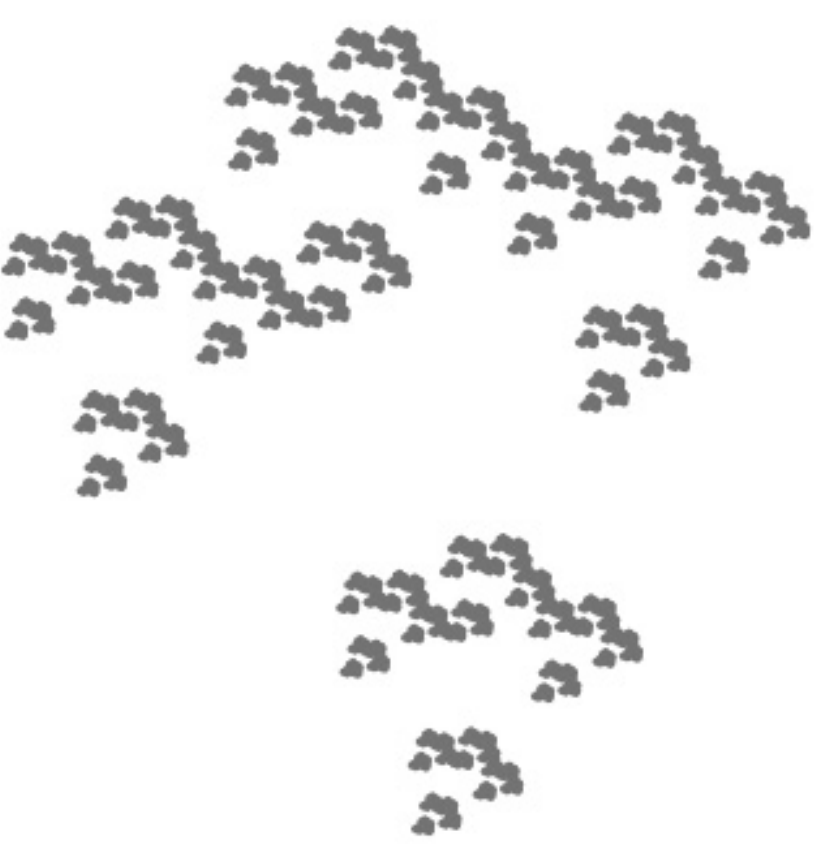}
 }
\caption{}\label{example5}
\end{figure}

Finally we consider a case of self-affine sets which is a  direct
analog of Example \ref{example1}. We make use the modified Lipschitz equivalence in Theorem \ref{th3.12} for self-affine sets.

\begin{Exa}  \label {ex5.5} { Let
$$
B=\left[\begin{array}{rr}
         0 & 1  \\
          -5 & -3
    \end{array} \right]
\quad \hbox {and} \quad {\mathcal D}=   \left\{\left[\begin{array}{rr} 0 \\
i\end{array}\right]: \ i=0,1,2,3,4 \right\}.
$$
Then $J$ is a self-affine tile; moreover it is known that $J$ is homeomorphic to the
unit disc, and it has six neighbors (Theorem 4.1 in \cite{LeLa07}):
\begin{equation} \label{eq5.1}
J \pm v,\ J \pm (Bv+2v),\ J \pm (Bv+3v)  \quad \hbox {with} \quad  v = \left[\begin{array}{rr} 0 \\
1\end{array}\right].
\end{equation}

If we let
\begin{equation*}
{\mathcal D}_1=\left\{\left[\begin{array}{rr} 0 \\
i\end{array}\right]: \
i=0,2,4\right\}\quad \hbox {and} \quad {\mathcal D}_2=\left\{\left[\begin{array}{rr} 0 \\
i\end{array}\right]: \ i=0,3,4\right\}
\end{equation*}
and let $K_1$ and $K_2$ be the corresponding self-affine sets. Then
the augmented trees are simple, $K_1$ is dust-like and  $K_1\simeq
K_2$ (see Figure \ref{example6}).}
\end{Exa}

\medskip

\begin{figure} [h]
  \centering
  \subfigure[]{
  \includegraphics[width=5cm]{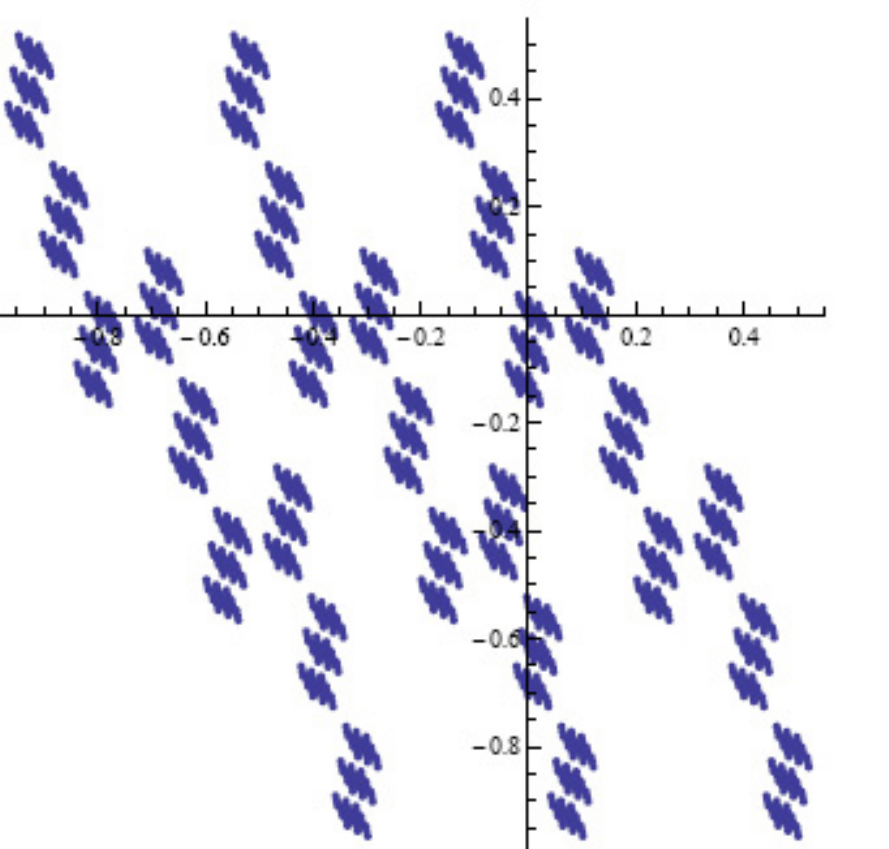}
 }
 \qquad
 \subfigure[]{
  \includegraphics[width=5cm]{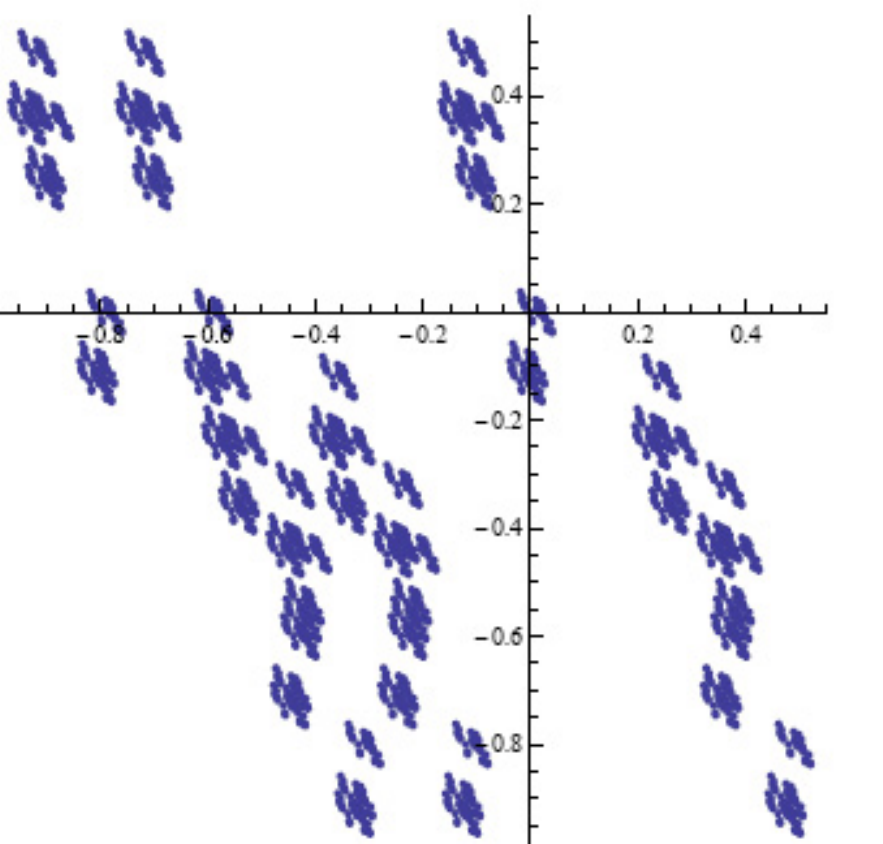}
 }
\caption{}\label{example6}
\end{figure}

Indeed, it is easy to see that the IFS for $K_1$ is strongly
separated. As the moduli of the eigenvalues of $B$ are equal to
$\sqrt{5}$,  by Theorem 1.3 and Corollary 3.2 of \cite{HeLa08},
$\dim_HK_1=\dim_H^w K_1= 2\log 3/ \log 5.$

For $K_2$, we let $S_j(x)=B^{-1}(x+d_j),\ d_j\in {\mathcal D}_2$ be
the IFS.
 Note that the tile $J$ satisfies $BJ= J + {\mathcal D}$, we will use this together with the neighborhood relation of $J$ to select
the corresponding ones for ${\mathcal D}_2$.  Let ${\mathscr T}_1 =
\{o\}$, we observe that $(J+3v)\cap (J+4v)\ne \emptyset$
(equivalently, $J\cap (J+v)\ne\emptyset$), hence there is only one
horizontal edge $2\sim 3$ in the first level of the augmented tree,
let ${\mathscr T}_2 = [\{2, 3\}]$. To find the next connected class
from ${\mathscr T}_2$, we note that the connectedness of
$\{S_{2j}(J), S_{3j}(J)\}_{j=1}^3$ is the same as (with a $B^2$
multiple)
$$
 J, \  (J+3v), \ (J+4v)\quad  \hbox {and} \quad
(J+Bv), \ (J+Bv+3v), \ (J+Bv+4v)
$$
(which are the cells of $B \big(J\cup (J+v)\big )$ corresponding to
${\mathcal D}_2$).  By (\ref {eq5.1}) (use the neighbors $J \pm v$
and $J \pm (Bv+3v)$),  we conclude that $\{2,3\}$ generates three
components
$$
\{21,32,33\},    \  \{22,23\},\  \{31\}.
$$
The first one belongs to a new class, for the iteration, we selected
from  $B\big(J\cup (J+Bv+3v)\cup(J+Bv+4v)\big)$ those cells
corresponding to ${\mathcal D}_2$ and use Hamilton-Cayley theorem to
obtain
$$
J, \  (J+3v), \ (J+4v);  \ (J-5v), \  (J-2v), \ (J-v);  \ (J+Bv-5v), \  (J+Bv-2v), \ (J+Bv-v).
$$
Again the neighborhood relation in (\ref{eq5.1}) determines the components
$$
 \{211\},\
\{212,213\},\ \{321\},\ \{322,323\},\ \{331\},\ \{332,333\}.
$$
There are no new ones if we continue the iteration, hence the three connected classes are: ${\mathscr T}_1=[o],\
{\mathscr T}_2=[\{2,3\}]$ and ${\mathscr T}_3=[\{21,32,33\}]$. The
incidence matrix is
$$
A=\left[\begin{array}{rrr}
          1 &  1 &  0  \\
          1 &  1 &  1\\
          3 &  3 &  0
    \end{array} \right].
$$
Clearly $A$ is primitive, and $K_2\simeq K_1$ by Theorem \ref{th3.12} under the ultra-metric defined by $w$. \hfill $\Box$

\end{section}

\bigskip

\begin{section}{\bf Remarks and open questions}

We note that not all incidence matrices of simple augmented trees are rearrangeable. For
example, we let
\begin{equation} \label{eq6.1}
K=\frac{1}{5}K\cup \frac{1}{5}(-K+4)\cup\frac{1}{5}(K+4).
\end{equation}
(It is a modification of $K_2$ in Example \ref {example1} by putting
in a reflection in the middle term.) Then the incidence matrix is
{\small $\left[\begin{array}{rrrr}
          1 &  1  \\
          0 &  3
    \end{array} \right].
$} It is  not rearrangeable, and it is easy to show that $K$ is a
countable union of dust-like sets, but we do not know if it is
Lipschitz equivalent to a dust-like set (see Figure \ref{fig6.1}).

\medskip

\begin{figure}[h]
  \centering
  \includegraphics[width=5cm]{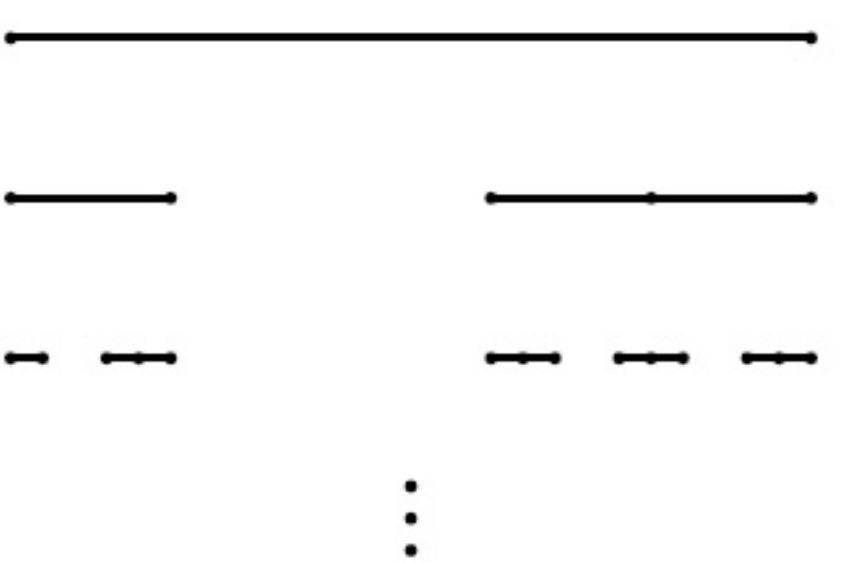}
\caption{ The totally disconnected $K$ defined by (\ref{eq6.1}).
}\label{fig6.1}
\end{figure}

\medskip

\noindent {\bf Q1}. Find conditions on the simple augmented tree
that describe the above situation. Also in view of Theorem
\ref{th3.6'}, is an irreducible incidence matrix $A$ rearrangeable?

\medskip

We have shown that if an augmented tree is simple, then the
self-similar set is totally disconnected (Proposition
\ref{additional prop}); also the converse is true on ${\Bbb R}$ under the OSC.  We ask

\medskip

\noindent {\bf Q2}. Suppose $K \subset {\Bbb R}^d$ is defined by (\ref {eq2.2}), if $K$
is totally disconnected, does it imply that augmented tree is simple, equivalently,  the horizontal
connected component of the augmented tree is uniformly bounded (with
or without assuming the OSC)?

\medskip

 In  recent literature, the investigations of the  Lipschitz equivalence are mainly for the {\it fractal cubes} of the form
 $K = \frac{1}{n}(K + {\mathcal D})$ where ${\mathcal D} \subset \{0, \dots , n-1\}^d$  (\cite{XiXi10}, \cite{XiXi11}).
 In Examples \ref{ex5.6} and \ref{ex5.5}, we bring in another classes of self-similar/self-affine sets $K$ that are based on the
 disk-like tiles on ${\Bbb R}^2$. Note that such tiles have six or eight neighbors \cite {BaWa01}, and for such $K$, the augmented trees depend very much on this neighborhood relationship. For the self-affine tiles generated by consecutive collinear digit sets, the disk-like property has been completely characterized in \cite{LeLa07} through the characteristic polynomials of the expanding matrices. This provides a wealth of source to study the topological property of such $K$. A general question is

\medskip

\noindent {\bf Q3.} For the  classes of self-similar/self-affine sets $K$
stated above, can we characterize the digit set ${\mathcal D}$ so
that $K$ is  totally disconnected?

\medskip

Also in  \cite {LaLuRa11}, the authors have made a head start to investigate the classification of the non-totally disconnected fractal squares.

\medskip

\noindent {\bf Q4.}  Can we put such classification into the
framework of augmented trees?

\medskip

In our consideration, we have restricted the IFS to have equal
contraction ratio. On the other hand, there is also interesting
study for different contraction ratios (e.g., \cite{RaRuWa10}, \cite{RuWaXi12}, \cite{XiRu07}). For
the augmented tree, the setup in \cite{LaWa09} is actually for
non-equal contraction ratios that the horizontal level $\Sigma^n$
can be replaced by the standard level set:  $\Lambda_n = \{{\bf i} =
i_1\cdots i_k:  r_{i_1} \cdots r_{i_k} \leq r^n < r_{i_1} \cdots
r_{i_{k-1}}\}$ where  $r_i$'s are the contraction ratios. More
generally, the authors have observed that the augmented trees can also be defined
for Moran fractals and the hyperbolic property still holds. It is
possible that the augmented trees can also be useful for the
Lipschitz equivalence of fractal sets in those cases.
\end{section}

\bigskip

\noindent{\bf Acknowledgements}: The authors would like to thank
Professors De-Jun Feng, Xing-Gang He, Hui Rao and  Xiang-Yang Wang
for many valuable comments and suggestions.

\end{document}